\documentclass[10pt]{article}
\usepackage[english]{babel}
\usepackage[utf8]{inputenc}
\usepackage[T1]{fontenc}
\usepackage{geometry}
\usepackage{amsthm,amsfonts}
\usepackage{amsmath}
\usepackage{mathabx}
\usepackage{enumitem}
\usepackage{titlesec}
\usepackage{empheq}
\usepackage{color}
\mathtoolsset{showonlyrefs=true}
\newcommand{\ce}{\mathbb{C}}

\newcommand{\er}{\mathbb{R}}

\newcommand{\ze}{\mathbb{Z}}
\newcommand{\vertiii}[1]{{\left\vert\kern-0.25ex\left\vert\kern-0.25ex\left\vert #1 
		\right\vert\kern-0.25ex\right\vert\kern-0.25ex\right\vert}}
\newcommand{\partere}{\mathrm{Re} }
\newtheorem{Teorema}{Theorem}[section]

\newtheorem{Rem}[Teorema]{Remark}
\newtheorem{Prop}[Teorema]{Proposition}
\newtheorem{Lemma}[Teorema]{Lemma}
\newtheorem{Coro}[Teorema]{Corollary}

\newtheorem{defi}{Definition}
\newtheorem{assumpt}{Assumption}

\newtheoremstyle{mytheoremstyle} % name
{\topsep}                    % Space above
{\topsep}                    % Space below
{}                   % Body font
{}                           % Indent amount
{\scshape}                   % Theorem head font
{.}                          % Punctuation after theorem head
{.5em}                       % Space after theorem head
{}  % Theorem head spec (can be left empty, meaning ‘normal’)

\theoremstyle{mytheoremstyle} \newtheorem{nota}{Remark}
\theoremstyle{mytheoremstyle} \newtheorem{exemplo}{Example}

\date{}
\author{Sim\~ao Correia, Filipe Oliveira, Jorge D. Silva}

\title{Mass-transfer instability of ground-states for Hamiltonian Schrödinger systems}

\begin{document}
	\maketitle
	\noindent
	\begin{abstract}
		We study generic semilinear Schrödinger systems which may be written in Hamiltonian form.
		In the presence of a single gauge invariance, the components of a solution may exchange mass between them while preserving the total mass. We exploit this feature to unravel new orbital instability results for ground-states. More precisely, we first derive a general instability criterion and then apply it to some well-known models arising in several physical contexts.
		 In particular, this \textit{mass-transfer instability} allows us to exhibit $L^2$-subcritical unstable ground-states.
		
		\vskip10pt
		\noindent\textbf{Keywords}: semilinear Schrödinger systems; bound-states; orbital instability.
		\vskip10pt
		\noindent\textbf{AMS Subject Classification 2010}: 35Q55, 35C08, 35B35.
	\end{abstract}
	
	\section{Introduction}
		In this paper, we consider general semilinear Schrödinger systems in pseudo-Hamiltonian form
\begin{equation}
\label{sistema}
\mathbf{u}_t=JH'(\mathbf{u})
\end{equation}
where $\mathbf{u}=(u_1,\dots,u_m)$, $u_j\,:\,\er^d\to \ce$, $J=\textrm{diag}\Big(\frac 1{i\lambda_j}\Big)$ with $\lambda_j\in\er$, $\lambda_j\neq 0$ and the Hamiltonian $H$ is of the form
$$H(\mathbf{u})=\frac 12\sum_j \int|\nabla u_j|^2+\frac{1}{2}N(\mathbf{u}),\quad N(\mathbf{u})=\sum_k \int n_k(\mathbf{u}),\quad n_k:\ce^m\to \er \mbox{ homogeneous of degree }\alpha_k.$$
There are many physically relevant models which may be written in this form. Such applications arise in several contexts, such as plasma physics, nonlinear optics or Bose-Einstein condensates, among others. The mathematical theory regarding the scalar case is, by now, well-established, covering local well-posedness, existence and stability of bound-states, global existence vs. finite time blow-up and scattering theory. In the last twenty years, the study of Schrödinger systems has become a very active field of research: on the one hand, the vector case presents a larger array of interesting physical models; on the other hand, one may observe new dynamical features that were not available in the scalar case. However, we believe that these new features have yet to be thoroughly explored.

We shall focus on the stability properties of bound-states, that is, solutions of the form
$$
\mathbf{u}(t)=(e^{i\omega_1 t}Q_1, \dots, e^{i\omega_m t}Q_m).
$$
In order for system \eqref{sistema} to admit bound-state solutions, we will require the existence of a gauge invariance:
\begin{equation}\label{invmass}
H(u_1e^{i\omega_1 t},\dots,u_me^{i\omega_mt })=H(u_1,\dots, u_m),\ \mbox{for all }\mathbf{u}\in (H^1(\er^d))^m,\ t\in \er.
\end{equation}
This condition is usually verified in any physically relevant model, since it is equivalent (see Appendix A) to the conservation of the total mass $$M(\mathbf{u}(t))=\frac 12 \sum_j\int  \lambda_j\omega_j|u_j(t)|^2, \quad \lambda_j\omega_j>0.$$
A direct computation shows that (the profiles of) bound-states, $\mathbf{Q}=(Q_1,\dots, Q_m)$, are precisely the critical points of the action functional
$$
S(\mathbf{u})=M(\mathbf{u})+H(\mathbf{u}).
$$

As in the scalar case, a special attention should be given to the bound-states with minimal action among all bound-states, the so-called ground-states. Indeed, it turns out that these solutions determine many dynamical properties of the full evolution problem. However, it is important to observe that this definition of ground-state is not very useful from a mathematical point of view, since it provides no information on the behavior of the action functional (even locally). For this reason, an important effort has to be made in order to show that ground-states are the solutions to some specific minimization problems (usually minimizing the action on a codimension one manifold in $(H^1(\er^d))^m$). Only then may one derive the numerous interesting properties regarding these solutions. In the general pseudo-Hamiltonian form, it is quite non-trivial to determine a suitable minimization problem: it depends on the specific power of the nonlinearities, their signs and also on the spatial dimension $d$. Since our goal is not to prove such a variational characterization, we shall define the set of minimal bound-states as
$$
\mathcal{B}_0 = \left\{ \mathbf{Q}\neq 0\mbox{ bound-state }: \mathbf{Q} \mbox{ is a local minimum of }S\mbox{ over a manifold }\mathcal{V}\subset(H^1(\er^d))^m\mbox{ of codimension }1  \right\},
$$
and study their instability. Since, in all known cases, minimal bound-states and ground-states coincide, we feel that this definition is in no way harmful to the validity of our work.

When studying the stability of minimal bound-states, it is essential to take into account the gauge and translation invariances. In fact, some simple arguments (see, for example, \cite[Section 8.3]{cazenave}) show that these invariances always induce an unstable behavior. Therefore, one should weaken the notion of stability: a bound-state is said to be \textit{orbitally stable} if, for any given initial data sufficiently close to it, the corresponding solution remains close  \textit{modulo gauge and translation invariances}.

For the nonlinear Schrödinger equation on $\er^d$
\begin{equation}\label{eq:nls}
iu_t + \Delta u + |u|^{p-1}u = 0,
\end{equation}
it is well-known that the ground-state is orbitally stable if and only if $p<4/d$ (corresponding to the $L^2$-subcritical case): on the one hand, if $p\ge 4/d$,  one may use a Virial-type argument to show that finite-time blow-up occurs for some initial datum arbitrarily close to the ground-state; on the other hand, if $p<4/d$, the ground-state can be shown to be (up to phase and translation) the minimizer of the action on a surface of constant total mass. The fact that both mass and action are preserved by the dynamical flow of \eqref{eq:nls} then implies the orbital stability. Evidently, for very particular systems of type \eqref{sistema}, the same dichotomy can be verified and the dynamical properties near the ground-state are the same as in the scalar case. Recalling our goal to find new dynamical behavior, instead of trying to disclose the optimal conditions that   ensure this precise threshold, we will analyze other properties, \textit{intrinsic to the vector-valued case}, that may induce instability.

In the seminal papers \cite{gss}, \cite{gss2}, the authors present very generic conditions that allow a complete characterization of the stability properties of ground-states. However, they assume a very precise knowledge of the linearized equation around the ground-state, specifically in what concerns the number of negative and null eigenvalues. In our context, the fine study of the linearized operator is quite challenging, especially due to the presence of multiple nonlinear terms and couplings. On the other hand, it was noticed in \cite{ribeiro} that the minimality of the ground-state $Q$ on the manifold
$$
\mathcal{V}=\left\{ u\in H^1(\er^d): \int |u|^{p+1} =  \int |Q|^{p+1} \right\},
$$ 
along with the existence of an unstable direction, is enough to prove orbital instablity. Heuristically, the minimality condition implies that the number of negative eigenvalues is either zero or one, thus reducing the stability problem to the existence of a negative direction. As we shall prove, this observation may be further extended to generic manifolds of codimension 1. Consequently, orbital instability will follow from the existence of a negative direction.

In this work, we study the conditions under which the curve
$$
\Gamma(t)=\Big(\gamma_1(t)\lambda^{\frac d2}(t)Q_1(\lambda(t) x),\dots,\gamma_m(t)\lambda^{\frac d2}(t)Q_m(\lambda(t)x)\Big),\quad \Gamma(0)=\mathbf{Q}, \quad M(\Gamma(t))=M(\mathbf{Q})$$
provides a direction for instability. The scaling factor $\lambda$ is connected to the Virial argument, while the coefficients $\gamma_j:[0,1]\to \er$ provide a way to exchange mass between components in such a way that the total mass is preserved. Consequently, any instability result obtained through the analysis of this curve shall be referred to as a \textit{mass-transfer instability}.
In a previous work \cite{CCFD}, we exploited this mechanism in a very concrete situation. Here, our goal is to derive a general criterion which may easily be applied to several semilinear Schrödinger systems at once.

At this point, it is important to notice that, if some other gauge invariance is present, then the choice of $\gamma_j$ is further restricted. In particular, if one has an invariance for each individual component, then the individual masses are conserved, thus preventing the mass-transfer mechanism (all $\gamma_j$ must be constant). Therefore, our results will be applied to systems presenting a single gauge invariance. 

Before we state our main results, we introduce a few notations and assumptions. To abbreviate, we write $e^{i\omega t}\mathbf{u} = (e^{i\omega_1 t}u_1,\dots, e^{i\omega_mt}u_m)$ and $\omega \mathbf{u}=(\omega_1 u_1, \dots, \omega_m u_m)$.
Define
$$
\beta_{j,k} = \mbox{ homogeneity degree of }n_k\mbox{ with respect to the }j^{th}\mbox{ component,}
$$
and assuming, without loss of generality, that the $m$-th component of the minimal bound-state $\mathbf{Q}$  is nonzero, we denote
$$
k_j=\frac{\lambda_j\omega_j \int |Q_j|^2}{\lambda_m\omega_m \int |Q_m|^2}.
$$
The first assumption concerns the initial value problem
\begin{equation}\label{ivp}
\mathbf{u}_t=JH'(\mathbf{u}),\quad \mathbf{u}(0)=\mathbf{u}_0.
\end{equation}
\begin{assumpt}[Local well-posedness]
	For $\mathbf{u}_0\in (H^1(\er^d))^m$, there exists $T=T(\|\mathbf{u}_0\|_{(H^1(\er^d))^m})$ and a unique $\mathbf{u}\in C([0,T], (H^1(\er^d))^m)$ solution of \eqref{ivp}. Moreover, we suppose that 
	$$H'\in C((H^1(\er^d))^m;\  (H^{-1}(\er^d))^m).
	$$
\end{assumpt}

In the vector-valued case, it may happen that the orbit of a bound-state is not closed. Our method does not cover this possibility. It is an interesting open problem to analyze orbital stability in this case. In the examples given below, the following assumption will be a consequence of the fact that $\omega_j\in\mathbb{Q}$ for all $j$.

\begin{assumpt}[Periodicity]
	The orbit $\{e^{i\omega t}\mathbf{Q}\}_{t\in \er}$ is closed.
\end{assumpt}

Finally, we require some regularity for the minimal bound-states, which may be verified using classical elliptic regularity bootstrap arguments.

\begin{assumpt}[Regularity]
	The bound-state $\mathbf{Q}$ satisfies $S''(\mathbf{Q}):(H^1(\er^d))^m \to  (H^{-1}(\er^d))^m$ and $x\cdot \nabla \mathbf{Q}\in (H^1(\er^d))^m$.
\end{assumpt}

\begin{Teorema}\label{teo:principal}
	Consider a real minimal bound-state $\mathbf{Q}$ of \eqref{sistema} and the symmetric matrix $A\in\mathcal{M}_{m\times m}$ given by
	\begin{displaymath}
	\begin{array}{llllllll}
	\displaystyle a_{0,0}&=&\displaystyle \dfrac d2\sum_k\Big(\dfrac{\alpha_k}2-1\Big)\Big(\dfrac{d\alpha_k}2-d-2\Big)\int n_k(\mathbf{Q})\\
	a_{0,j}&=&\displaystyle \sum_k \Big(\dfrac{d\alpha_k}4-\frac d2-1\Big)(\beta_{j,k}-k_j\beta_{m,k})\int n_k(\mathbf{Q})\\
	\displaystyle a_{j,j}&=&\displaystyle \sum_k\Big(\frac 12 k_j^2\beta_{m,k}(\beta_{m,k}-2)+\frac 12\beta_{j,k}(\beta_{j,k}-2)-k_j\beta_{j,k}\beta_{m,k}\Big)\int n_k(\mathbf{Q})\\
	a_{j,j_0}&=&\displaystyle\sum_k \frac 12\Big(k_jk_{j_0}\Big(\beta_{m,k}(\beta_{m,k}-2)\Big)+
	\displaystyle\beta_{j,k}\beta_{j_0,k}-\beta_{m,k}(k_j\beta_{j_0,k}+k_{j_0}\beta_{j,k})
	\Big)\int n_k(\mathbf{Q})
	\end{array}
	\end{displaymath}
	for $1\le j<m$ and $j\neq j_0$. If $A$ admits one negative eigenvalue then $\mathbf{Q}$ is orbitally unstable.
\end{Teorema}

\begin{nota}
	Even though it is not trivial to derive generic sufficient conditions for the existence of a negative eigenvalue for $A$, Theorem \ref{teo:principal} may be applied quite easily to any particular system, as we illustrate below. In fact, since this criterion is amenable to perturbations, there is no need to know  the \textit{exact} values of $\int n_k(\mathbf{Q})$. Therefore, in the cases where the exact derivation of a formula for the minimal bound-state is challenging, one may use a numerical approximation, thus allowing for computer-assisted proofs of orbital instability.
\end{nota}

Before we proceed to study concrete examples, we apply Theorem \ref{teo:principal} to derive some simple instability results. The first can be obtained directly from the analysis of the scaling parameter $\lambda(t)$ and thus it is just a generalization of the scalar $L^2$-supercritical instability.

\begin{Prop}[$L^2$-supercritical instability]\label{ucoro} Given $p>2+4/d$, suppose that $N$ admits the following decomposition: 
	$$N(\mathbf{u})=N_2(\mathbf{u})+N_{<p}(\mathbf{u})+N_{p}(\mathbf{u})+N_{>p}(\mathbf{u}),$$
	where 
	\begin{enumerate}
		\item $N_2$ is quadratic;
		\item $N_{<p}$ is the sum of homogeneities smaller than $p$ and is nonnegative;
		\item $N_p$ has homogeneity equal to $p$;
		\item $N_{>p}$ is the sum of homogeneities larger than $p$ and is nonpositive.
	\end{enumerate}
	Then any real minimal bound-state is orbitally unstable.
\end{Prop}

We now focus on the $L^2$-critical case:
$$N(\mathbf{u})=N_2(\mathbf{u})+N_p(\mathbf{u}),$$
where $N_2$ is quadratic and $N_p$ is homogeneous of degree $p=2+4/d$.

\begin{Prop}[$L^2$-critical instability I]\label{L2critical}
	Suppose that
	$$N_2(\mathbf{u})=\sum_{j=1}^m c_j \int |u_j|^2,\quad c_j\in\er. $$
	If $c_1\lambda_m\omega_m\neq c_m\lambda_1\omega_1$, then all real minimal bound-states $\mathbf{Q}$ with $Q_1, Q_m\neq 0$ are orbitally unstable.
\end{Prop}

\begin{Prop}[$L^2$-critical instability II]\label{L2critical2}
	Suppose that
	$$N_2(\mathbf{u})= \lambda\partere \int u_1\bar{u}_m,\quad \lambda\neq 0. $$
	If a real minimal bound-state $\mathbf{Q}$ satisfies $\lambda_1\omega_1\int Q_1^2 \neq \lambda_m\omega_m\int Q_m^2$ and $\int Q_1Q_m\neq 0$, then $\mathbf{Q}$ is orbitally unstable.
\end{Prop}

\begin{nota}
	The above results provide criteria depending on the first and last components of the bound-state. Evidently, since one may choose the order of the components, this causes no loss in generality. Moreover, the scaling parameter by itself is not sufficient to conclude instability, which means that these results are truly intrinsic to the vector-valued case.
\end{nota}

Finally, by continuity, we realize that instability of minimal bound-states may even occur in $L^2$-subcritical cases:
\begin{Coro}[Instability in almost $L^2$-critical cases]
	\label{corL2critical}
	Suppose that 
	$$N(\mathbf{u})=N_2(\mathbf{u})+N_p(\mathbf{u}),$$
	where $N_p$ is homogeneous of degree $p$ and $N_2$ is as in Proposition \ref{L2critical}. Assume that there exists a continuous curve of real minimal bound-states $p\mapsto \mathbf{Q}(p)$ around $p_0=2+4/d$. If the hypotheses of Proposition \ref{L2critical} are verified for $p_0$, then all bound-states $\mathbf{Q}(p)$,  $|p-p_0|\ll 1$, are orbitally unstable. The result remains valid if one replaces Proposition \ref{L2critical} by Proposition \ref{L2critical2}.
\end{Coro}

%
%The conserved Hamiltonian is given by
%$$H(\mathbf{u})=\frac 12\sum_j \int|\nabla u_j|^2+\frac 12\sum_k \int n_k(\mathbf{u}):=K(\mathbf{u})+N(\mathbf{u}),$$
%where $ n_k(\mathbf{u})$ is an homogeneous term of degree $2\le \alpha_k\le 2d/(d-2)^+$,
%$$|n_k(\mathbf{u})|=\prod_j|u_j|^{\beta_{j,k}},\qquad \sum_j \beta_{j,k}=\alpha_k.$$
%We will further assume that for some $\omega=(\omega_1,\dots,\omega_M)$,
%\begin{equation}
%\label{invmass} 
%H(u_1e^{i\omega_1},\dots,u_Me^{i\omega_M})=H(u_1,\dots, u_M),\ \mbox{for all }\mathbf{u}\in (H^1(\er^d))^M.
%\end{equation}
%As we will see later, this condition will imply the conservation of mass.
%\vskip15pt
%
%\noindent\textbf{Notation.} To abbreviate, we write $e^{i\omega t}\mathbf{u} = (e^{i\omega_1 t}u_1,\dots, e^{i\omega_Mt}u_m)$ and $\omega \mathbf{u}=(\omega_1 u_1, \dots, \omega_M u_M)$.
%
%\bigskip

\bigskip

\noindent
For illustrative purposes, we will apply these results to the following concrete models:
\begin{exemplo}[Quadratic Schr\"odinger system I]
		\end{exemplo}
	\begin{equation}\label{quadeq}
	\left\{\begin{array}{llll}
	iu_t+\Delta u+\overline{u}v=0\\
	i\sigma v_t+\Delta v-\beta v+\dfrac 12 u^2=0,\quad (x,t)\in \er^d\times \er.
	\end{array}\right.
	\end{equation}	
	This model governs the resonant interaction between waves propagating in a $\chi^{(2)}$ dispersive medium in several physical contexts, such as magneto-hydrodynamics or nonlinear optics (see for instance \cite{q10}, \cite{q11},\cite{q12}). In \cite{Saut},  the authors prove the existence of ground-state solutions of the form $(e^{i\omega t}Q_1(x),e^{i\omega t}Q_2(x))$ for this system by minimizing the action over the manifold
	$$
	\mathcal{V}_\lambda=\left\{ (u,v)\in (H^1(\er^d))^2: \mbox{Re }\int u^2\bar{v} = \lambda  \right\},
	$$
	for some specific $\lambda\in\er$. Moreover, they show orbital stability in the $L^2$-subcritical dimension $d=2$ when $\beta=0$.
	
	\noindent
	We may recover the instability results that we derived in \cite{CCFD} for the $L^2$-(super)critical cases: the ground-state $\mathbf{Q}=(Q_1,Q_2)$ is orbitally unstable if
	\begin{enumerate}
		\item $d\ge 5$ (Corollary \ref{ucoro});
		\item $d=4$ and $\beta\neq 0$ (Proposition \ref{L2critical}).
	\end{enumerate}
%	
%	We may recover the instability results that we derived in \cite{CCFD} for the $L^2$-(super)critical cases:
%	for $d\geq 5$, $Q=(Q_1,Q_2)$ is orbitally unstable (Proposition \ref{ucoro});
%	if $d=4$, the critical dimension, we get the same conclusion as long as $\beta\neq 0$ (Corollary \ref{corL2critical}).\\
	Also, in the synchronous case $\beta=\omega(1-2\sigma)$, the ground-state can be computed explicitly. Indeed, one has $$\mathbf{Q}=\mathbf{Q}_{sync}=\left(\sqrt{\dfrac 23}q,\dfrac 1{\sqrt{3}}q\right),$$ where $q$ is the ground-state of 
	$$-\omega q+\Delta q+\frac{1}{\sqrt{3}}q^2=0.$$
	In this situation, in the subcritical cases $d\leq 3$, Theorem \ref{teo:principal} yields that for $\sigma\gg1$, $\mathbf{Q}_{sync}$ is orbitally unstable:
	\begin{Prop}\label{prop:subL2quad2D}Let $d\leq 3$ and $\beta=\omega(1-2\sigma)$. Define $\sigma_0$ as the positive root of
		$$
		3(4-d)(1+4\sigma)-(1-2\sigma)^2=0.
		$$
	Then, for $\sigma>\sigma_0$, $\mathbf{Q}_{sync}$ is orbitally unstable by the flow of \eqref{quadeq}.
	\end{Prop}
	
	\medskip
	\noindent
	To the best of our knowledge, there are no examples in the literature of such $L^2$-subcritical unstable ground-states for Schr\"odinger-type coupled systems. We will exhibit another one in the next example:

\begin{exemplo}[Quadratic Schr\"odinger system II]
\end{exemplo}
	\begin{equation}\label{eqademir1}
\left\{\begin{array}{llll}
3iu_t+\Delta u-\beta_1u=-vw\\
2iv_t+\Delta v-\beta v=-\dfrac 12w^2-u\overline{w}\\
iw_t+\Delta w-w=-v\overline{w}-u\overline{v},\quad (x,t)\in \er^d\times \er.
\end{array}\right.
\end{equation}

\medskip

\noindent
This system is a generalization of the previous one for three-wave interactions in the framework of optical fiber systems.  (see \cite{Kivshar}). Both quadratic systems arise when second order nonlinear processes, such as second harmonic generation, are taken into account. Under some conditions on $\beta, \beta_1$, the existence of ground-states may be achieved through the minimization of the action on
	$$
\mathcal{V}_\lambda=\left\{ (u,v)\in (H^1(\er^d))^2: \mbox{Re }\int \left(2\bar{u}vw + w^2\bar{v}\right) = \lambda  \right\},
$$
for a well-chosen $\lambda\in\er$.

\bigskip

\noindent
Again, in the synchronous case $\beta_1=-7$ and $\beta=-2$, it is possible to compute explicitly the ground-state $\mathbf{Q}_{sync}:=(aq,bq,cq)$, where $(a,b,c)\in \mathbb{S}^2$ and $q$ is the ground-state of
$$-2q+\Delta q+\dfrac 16(\sqrt{3}+\sqrt{5})q^2=0.$$
In this framework, Theorem \ref{teo:principal} yields the following result:
\begin{Prop}\label{prop:subL2quad3D}Let $d=1$, $\beta_1=-7$ and $\beta=-2$. Then the ground-state $\mathbf{Q}_{sync}$ is orbitally unstable by the flow of \eqref{eqademir1}.
\end{Prop}
Notice that our method does not seem to allow any conclusions in the subcritical dimensions $d=2$ and $d=3$. 
\medskip
\begin{nota}
Interestingly enough, in \cite{Pastor}, the author showed that $\mathbf{Q}_{sync}$ is spectrally stable in dimension $d=1$, \textit{i.e.}, that the linearized operator around the ground-state does not have any negative eigenvalues. The spectral stability of the ground-state only implies that there are no exponentially diverging solutions in its neighborhood (and thus it is a weaker notion than that of orbital stability). Therefore there is no contradiction with Proposition \ref{prop:subL2quad3D}.
\end{nota}
%\noindent
%Interestingly enough, in \cite{Pastor}, the author showed the global wellposedness of \eqref{eqademir1} in $L^2(\er^d)$ and $H^1(\er^d)$ and the spectral stability of $Q_{sync}$ in dimension $d=1$.\\
%Also, notice that our method does not seem to allow any conclusions in the subcritical dimensions $d=2$ and $d=3$. 

\begin{exemplo}[Cubic Schr\"odinger system I]
\end{exemplo}
	\begin{equation}
	\label{cubic}
\left\{\begin{array}{llll}
\displaystyle iu_t+\Delta u-u+\Big(\frac 19|u|^2+2|w|^2\Big)u+\frac 13\overline{u}^2w=0\\
\\
\displaystyle 	i\sigma w_t+\Delta w-\mu w+(9|w|^2+2|u|^2)w+\frac 19u^3=0,\quad (x,t)\in \er^d\times \er.
\end{array}\right.
\end{equation}
\noindent
This system, derived by Sammut \textit{et al.} in \cite{Sammut}, models the resonant interaction between a monochromatic beam of frequency $\omega$ propagating in a Kerr $\chi^{(3)}$ material and its third harmonic. The third-harmonic generation leads to features typical of $\chi^{(2)}$ media. 
%For $d\leq 3$ (and under some conditions on the initial mass $M$ and energy $H$ in the $L^2$-(super)critical dimensions $d=2,3$), it was shown in \cite{FA} the global well-posedness of \eqref{cubic} in $(H^1(\er^d))^2$. Furthermore, we showed the existence of a ground-state $Q$ for $\sigma, \mu>0$ and $\omega>\max\{-1,-\frac{\mu}{3\sigma}\}$. This was achieved by minimizing the action $S=M+\omega H$ over a Nehari manifold.
In \cite{FA}, the existence of a ground-state $\mathbf{Q}$ for $\sigma, \mu>0$ and $\omega>\max\{-1,-\mu/{3\sigma}\}$ was proved by minimizing the action over the Nehari manifold
$$
\mathcal{V}=\left\{ (u,v)\in (H^1(\er^d))^2\setminus\{(0,0)\}: \langle S'(u,v), (u,v)\rangle_{H^{-1}\times H^1}=0 \right\}.
$$
As in the quadratic cases, we conclude that the ground-state solution is orbitally unstable if:
	\begin{enumerate}
		\item $d\ge 3$ (Corollary \ref{ucoro});
		\item $d=2$ and $\mu\neq 3\sigma$ (Proposition \ref{L2critical}).
	\end{enumerate}
%As for the case of the quadratic systems, Proposition \ref{ucoro} yields the orbital instability of $Q$ in dimension $d=3$. Also,  Corollary \ref{corL2critical} shows the orbital instability of $Q$ in the critical dimension $d=2$ as long as $\mu\neq 3\sigma$.   

\begin{exemplo}[Cubic Schr\"odinger system II]
	\end{exemplo}
	\begin{equation}\label{eq:rabi}
	\left\{\begin{array}{llll}
	iu_t + \Delta u +  \lambda v + k_{11}|u|^2u + k_{12}|v|^2u =0\\
	iv_t + \Delta v + \lambda u + k_{12}|v|^2u + k_{22}|v|^2v=0,\quad (x,t)\in \er^2\times \er.
	\end{array}\right.
	\end{equation}
This system models a two-component Bose-Einstein condensate irradiated by an external electromagnetic field with no trapping potential and with Rabi frequency $\lambda$ (\cite{Rabi1},\cite{Rabi2}). The existence of bound-states $(u,v)=e^{i\omega t}(P,Q)$, $\omega>\lambda$, can be achieved by minimizing the action on the manifold
$$
\mathcal{V}_\lambda=\left\{ (u,v)\in (H^1(\er^d))^2: \int \left(k_{11}|u|^4 + 2k_{12}|u|^2|v|^2 + k_{22}|v|^4\right) = \lambda  \right\}, \quad \lambda>0.
$$
%$$
%\mbox{Minimize }\frac{1}{2}\int |\nabla u|^2 + |\nabla v|^2 - \lambda \partere \int u\bar{v} + \frac{\omega}{2} (|u|^2 + |v|^2)$$$$ \mbox{ subject to }k_{11}\int |u|^4 + 2k_{12}\int |u|^2|v|^2 + k_{22}\int |v|^4=\mbox{constant}.
%$$
If $k_{12}>0$, a standard application of the Schwarz symmetrization reveals that $P, Q$ are real and radially decreasing. Since neither $P$ nor $Q$ can be zero, $\int PQ\neq 0$. Applying Proposition \ref{L2critical2}, we derive the following:
\begin{Prop}
	Set $\lambda\neq 0$ and $k_{12}>0$. If the ground-state $(P,Q)$ satisfies $\int P^2\neq \int Q^2$, then it is orbitally unstable.
\end{Prop}
%\begin{Propriedade} Assume $\omega>\lambda$. Then \eqref{eq:rabi} admits ground-state solutions of the form $$(u(x,t),v(x,t))=(e^{i\omega t}Q_1(x),e^{i\omega t}Q_2(x))$$ (if $k_{12}>0$, $Q_1,Q_2$ are real and radially decreasing). \\
%If $\lambda\neq 0$ and $\displaystyle \int Q_1^2\neq \int Q_2^2$, the ground-state is orbitally unstable.
%\end{Propriedade}

\bigskip

\noindent
As it should be clear from the above examples, \textit{the existence of linear terms may induce unstable behavior through the mass-transfer instability}. When there is a gauge invariance for each individual component, these linear terms may be absorbed using a simple change of variables. In the cases considered, however, the presence of a single, complete gauge invariance prevents this procedure.

\section{Weak instability of ground-states}
We recall the action functional
$$S(\mathbf{u}):=M(\mathbf{u})+H(\mathbf{u})$$
and the set of minimal bound-states
$$
\mathcal{B}_0 = \left\{ \mathbf{Q}\neq 0\mbox{ bound-state }: \mathbf{Q} \mbox{ is a local minimum of }S\mbox{ over a manifold }\mathcal{V}\subset(H^1(\er^d))^m\mbox{ of codimension }1  \right\}.
$$
\begin{nota}
	Observe that $\mathbf{Q}$ is a bound-state if and only if $S'(\mathbf{Q})=0$. Indeed, taking $\mathbf{u}=e^{i\omega t}\mathbf{Q}=(e^{iw_1t}Q_1(x),\dots e^{iw_mt}Q_m(x))$, 
	$$\mathbf{u}_t=JH'(\mathbf{u})\Leftrightarrow \forall j,\,iw_je^{i\omega_j t}Q_j=\frac 1{i\lambda_j}e^{i\omega_j t}\frac{\partial H}{\partial u_j}(\mathbf{Q})\Leftrightarrow \forall j,\frac{\partial }{\partial u_j}(M+H)(\mathbf{Q})=0\Leftrightarrow S'(\mathbf{Q})=0.$$
\end{nota}
%We recall the action
%$$S(\mathbf{u}):=M(\mathbf{u})+H(\mathbf{u})$$
%and define the set of its critical points, the bound-states of \eqref{sistema}:
%
%$$\mathcal{B}=\left\{\mathbf{Q}\in(H^1(\er^d))^M\setminus\{0\}\,:\,S'(\mathbf{Q})=0\right\}.$$
%Notice that $e^{i\omega t}\mathbf{Q}$  is a solution of \eqref{sistema} if and only if $\mathbf{Q}\in \mathcal{B}.$ Indeed, by the invariance \eqref{invmass} and taking $\mathbf{u}=e^{i\omega t}\mathbf{Q}=(e^{iw_1t}Q_1(x),\dots e^{iw_Mt}Q_M(x))$, 
%$$\mathbf{u}_t=J(H'(\mathbf{u}))\Leftrightarrow \forall j,\,iw_je^{i\omega_j t}Q_j=\frac 1{i\lambda_j}\frac{\partial H}{\partial u_j}(\mathbf{Q})\Leftrightarrow \forall j,\frac{\partial }{\partial u_j}(M+H)(\mathbf{Q})=0.$$
As previously explained, one needs to study the stability properties modulo the 
%The set $\mathcal{B}$ possesses the natural 
invariance
$$\mathbf{Q}\mapsto f(\theta,y)\mathbf{Q}:=(e^{is\omega_1}Q_1(\cdot+y),\dots,e^{is\omega_m}Q_m(\cdot+y)),\quad \theta\in \er,\ y\in \er^d.$$
We therefore define the orbit of a bound-state $\mathbf{Q}$ as
$$\mathcal{O}_\mathbf{Q}=\{f(\theta,y)\mathbf{Q},\,(\theta,y)\in\er\times \er^d\}.$$
Since the total mass is invariant under both the dynamical flow and  the gauge and translation invariances, a great part of the analysis shall be performed on 
$$\mathcal{M}_\mathbf{Q}=\{\mathbf{u}\in (L^2(\er^d))^m\,:\,M(\mathbf{u})=M(\mathbf{Q})\}.$$
To simplify the exposition, we will drop the subscript and write $\mathcal{M}=\mathcal{M}_\mathbf{Q}$.

%Defining, for a given bound-state, the orbit
%$$\mathcal{O}_\mathbf{Q}=\{f(\omega,y)\mathbf{Q},\,(\omega,y)\in\er^M\times \er\},$$
%it is clear that the solution $\mathbf{u}(t)$ of \eqref{sistema} for initial data $\mathbf{u}(0)=\mathbf{Q}$ is global in time and remains in the orbit: for all $t$, $\mathbf{u}(t)\in \mathcal{O}_\mathbf{Q}$.

%\noindent
%$$
%\mathcal{B}_0 = \left\{ \mathbf{Q}\in \mathcal{B}: \mathbf{Q} \mbox{ is a local minimum of the action over a manifold }\mathcal{V}\mbox{ of codimension }1  \right\}.
%$$
%

%From all bound-states we single out the so-called ground-states, minimizers of the action $S$ among all bound-states: by denoting $\mathcal{G}$ the set of all ground-states,
%$$\mathcal{G}=\{\mathbf{Q}\in\mathcal{B}\,:\,\forall \mathbf{Q}\in \mathcal{B},\,S(\mathbf{Q})\leq S(\mathbf{Q})\}.$$
\vskip10pt 
 Finally, we say that a bound-state $\mathbf{Q}$ is orbitally unstable if there exist solutions of $\eqref{sistema}$ with initial data near $\mathbf{Q}$ which move away from the orbit of $\mathbf{Q}$. More precisely:

\begin{defi}[Orbital instability]
A bound-state $\mathbf{Q}$ is said to be orbitally unstable by the flow of \eqref{sistema} if there exist $\epsilon>0$ and a sequence $(\mathbf{u}_0)_k\to \mathbf{Q}$ in $\in (H^1(\er^d))^m$ such that 
the solution $\mathbf{u}_k$ of \eqref{sistema} with initial data $(\mathbf{u}_0)_k$ satisfies
		$$T_k:=\sup\{t\,:\,d(\mathbf{u}_k(t),\mathcal{O}_\mathbf{Q}))<\epsilon\}<+\infty.$$
\end{defi}

%In this section we will show that, under some certain circumstances the real ground-states of \eqref{sistema} are orbitally unstable.

%\medskip
%
%\noindent
%Given a bound-state $\mathbf{Q}$, we define the manifold
%$$\mathcal{M}=\{\mathbf{u}\in (H^1(\er^d))^M\,:\,M(\mathbf{u})=M(\mathbf{Q})\}.$$

\subsection{An orbital instability condition}

\noindent
In this paragraph, we prove the following orbital instability condition:
\begin{Teorema}
\label{gon}	
	 Let $\mathbf{Q}\in \mathcal{B}_0$. If there exists $\Psi\in T_\mathbf{Q}\mathcal{M}$ such that
	\begin{itemize}
		\item $J^{-1}\Psi$ is $L^2$-orthogonal to $J^{-1}\omega\mathbf{Q}$ and to $\partial_{x_j}\mathbf{Q}$ for all $j=1, \dots, d$;
		\item For all $j=1,\dots, d$, $J^{-1}\omega \mathbf{Q}$ and  $\partial_{x_j}\mathbf{Q}$ are linearly independent;
		\item $\langle S''(\mathbf{Q})\Psi,\Psi\rangle_{H^{-1}\times H^1}<0$,
	\end{itemize}
	then $\mathbf{Q}$ is orbitally unstable.
\end{Teorema}

We follow the main ideas  for the scalar case  presented in \cite{SS}, in the context of the nonlinear Schr\"odinger equation. We extend the arguments of Shatah and Strauss to more general semilinear Schrödinger systems and for bound-states in $\mathcal{B}_0$.

Throughout this section, $\mathbf{Q}\in\mathcal{B}_0$ will be fixed. We define the ${L^2}$-orthogonal hyperplane to the orbit as 
$$L=\{\mathbf{w}\in (L^2(\er^d))^m\,:\,\mathbf{w}\perp_{L^2} J^{-1}\omega\mathbf{Q}\textrm{ and }\forall j,\,\mathbf{w}\perp_{L^2}\partial_{x_j}\mathbf{Q}\}$$
and define the neighborhoods
$$L_{\delta}=\{\mathbf{w}\in (H^1(\er^d))^m\cap (\mathbf{Q}+L)\,:\,\|\mathbf{w}-\mathbf{Q}\|_{(H^1(\er^d))^m}<\delta\} \mbox{ and } \mathcal{O}_{\mathbf{Q},\delta}=\{f(\theta,y)L_{\delta},\,(\theta,y)\in\er^{d+1}\}.$$
% obtained by minimizing the action over a general hypersurface $\mathcal{V}$.\\
%\noindent
%For $\mathbf{Q}\in\mathcal{B}_0$, we put $L=\{W\in H^1(\er^d)\,:\,W\perp i\omega\mathbf{Q}\textrm{ and }\forall j,\,W\perp\partial_{x_j}\mathbf{Q}\}$.\\ Also, for $\delta>0$, let  $L_{\delta}=\{W\in L\,:\,\|W-\mathbf{Q}\|_{H^1}<\delta\}$ and $\mathcal{O}_{\mathbf{Q},\delta}=\{f(s,y)L_{\delta},\,(s,y)\in\er^{d+1}\}$.

\begin{Lemma}
	\label{Lema42}
There exists $\delta>0$ such that  $\forall (\theta,y)\in (0,2\pi/\tilde{\omega})\times \er^d$, 
$$L_{\delta}\cap f(\theta,y)L_{\delta}=\emptyset.$$
Here, $2\pi/\tilde{\omega}$ is the minimal period of $\mathbf{u}=e^{i\omega t}\mathbf{Q}$. As a consequence, $\mathcal{O}_{\mathbf{Q},\delta}$ is an open set in $(H^1(\er^d))^m$.
\end{Lemma}
\begin{nota}
	The existence of  $\tilde{\omega}$ is ensured by Assumption 2.
\end{nota}
%\begin{Lemma}
%	\label{Lema42}
%	There exists $\delta>0$ such that $L_{\delta}$ is an open subset of $(H^1(\er^d))^M$ and   $\forall (s,y)\in \Big(0,\dfrac{2\pi}{\tilde{\omega}}\Big)\times \er^d$, 
%	$$L_{\delta}\cap f(s,y)L_{\delta}=\emptyset.$$
%	Here, $\tilde{\omega}=LCM(\omega_1,\dots,\omega_n)$.
%\end{Lemma}
\begin{proof}
	\noindent
	Let  $\delta>0$ and $s_0\equiv 0 \mod 2\pi/\tilde{\omega}$. We set
	$$
	F(\theta,y,\mathbf{z},\mathbf{w})=\mathbf{z}-f(s,y)\mathbf{w}, \quad (\theta,y,\mathbf{z},\mathbf{w})\in\er\times \er^d\times L_{\delta}\times L_{\delta}.
	$$
% put, for $(s,y,Z,W)\in\er\times \er^d\times L_{\delta}\times L_{\delta}$, $F(s,y,Z,W)=Z-f(s,y)W.$\\
	Observe that $F(0,0,\mathbf{Q},\mathbf{Q})=0$ and that the Jacobian matrix of $F$ at $(0,0,\mathbf{Q},\mathbf{Q})$ with respect to $s,\mathbf{w}$ and $\mathbf{z}$ is given by
	$$
	\left[\begin{array}{llllll}
	J^{-1}\omega\mathbf{Q}&0&0&\dots&0\\
	0&\partial_{x_1}\mathbf{Q}&0&\dots&0\\
	0&0&\dots&\dots&0\\
	0&0&\dots&\partial_{x_n}\mathbf{Q}&0\\
	0&0&0&\dots&\mathbf{Q}
	\end{array}
	\right].$$
	Since $\mathbf{Q}\in \langle J^{-1}\omega \mathbf{Q}, \partial_{x_1}\mathbf{Q},\dots,\partial_{x_n}\mathbf{Q}\rangle^{\perp}$ and $J^{-1}\omega\mathbf{Q}, \partial_{x_1}\mathbf{Q},\dots,\partial_{x_n}\mathbf{Q}$ are linearly independent, this matrix is invertible. Hence, by the Implicit Function Theorem, for $\delta$ small enough and $\mathbf{z}\in L_{\delta}$, there exists a unique $(\theta,y,\mathbf{w})$, small in $(\ze/(2\pi/\tilde{\omega})\ze)\times\er^d\times L_{\delta}$, such that $\mathbf{z}=f(\theta,y)\mathbf{w}$. The choice is obviously $(\theta,y,\mathbf{w})=(0,0,\mathbf{z})$ and the Lemma is proved for $\delta$ small enough and $(\theta,y)$ in a neighbourhood of $(0,0)\in (\ze/(2\pi/\tilde{\omega})\ze)\times\er.$    
	
	\bigskip
	
	\noindent
	Now, by contradiction, let us assume the existence of sequences $\delta_n\to 0$, $(\theta_n,y_n)\in [0,2\pi/\tilde{\omega})\times \er^d$ and  $(\mathbf{w}_n)_{n\in \mathbb{N}}\in L_{\delta_n}$ such that $\mathbf{z}_n=f(\theta_n,y_n) \mathbf{w}_n\in L_{\delta_n}$. By definition of $L_{\delta}$, $\mathbf{w}_n\to \mathbf{Q}$  and $\mathbf{z}_n\to \mathbf{Q}$ in $H^1(\er^d)$. Also, since $(\theta_n)_{n\in\mathbb{N}}$ is bounded, we may assume that $\theta_n\to \theta_0\in [0,2\pi/\tilde{\omega}]$.
	
	\bigskip
	
	\noindent
	If $(y_n)_{n\in\mathbb{N}}$ is also bounded, then $y_n\to y_0$ up to a subsequence. Since $L_{\delta}\cap f(\theta,y)L_{\delta}=\emptyset$ for small $\theta,y$ and $\delta$, $(\theta_0,y_0)\notin \{0, 2\pi/\tilde{\omega}\}\times\{0\}$. Also,
	$$\|\mathbf{Q}-\mathbf{z}_n\|_{L^2}=\|\mathbf{Q}-f(\theta_n,y_n)\mathbf{w}_n\|_{L^2}=\|f(-\theta_n,-y_n)\mathbf{Q}-\mathbf{w}_n\|_{L^2}\to \|f(-\theta_0,-y_0)\mathbf{Q}-\mathbf{Q}\|_{L^2}=0.$$  
	Hence $\mathbf{Q}=f(\theta_0,y_0)\mathbf{Q}$, which contradicts the fact that 
	$(\theta_0,y_0)\notin \{0, 2\pi/\tilde{\omega}\}\times\{0\}$.
	
	\medskip
	
	\noindent 
	On the other hand, if $(y_n)_{n\in\mathbb{N}}$ is unbounded,  $|y_n|\to +\infty$ up to a subsequence. Then,
	$$\|\mathbf{Q}-\mathbf{z}_n\|_{L^2}^2=\|\mathbf{Q}\|_{L^2}^2+\|\mathbf{w}_n\|_{L^2}^2-2\int f(-\theta_n,-y_n)\mathbf{Q}\mathbf{w}_n$$
	\begin{equation}
	\label{contra}
	=\|\mathbf{Q}\|_{L^2}^2+\|\mathbf{w}_n\|_{L^2}^2-2\int f(-\theta_n,-y_n)\mathbf{Q}(\mathbf{w}_n-\mathbf{Q})-2\int f(-\theta_n,-y_n)\mathbf{Q}\mathbf{Q}.
	\end{equation}
	It is clear that $\|\mathbf{w}_n\|_{L^2}\to \|\mathbf{Q}\|_{L^2}$ and $\displaystyle \int f(-\theta_n,-y_n)\mathbf{Q}(\mathbf{w}_n-\mathbf{Q})\to 0$. Furthermore,
	$$\int f(-\theta_n,-y_n)\mathbf{Q}\mathbf{Q}=e^{-i\theta_n\omega}\int \mathbf{Q}(x-y_n)\mathbf{Q}(x)dx\to 0$$
	since $\mathbf{Q}\in (L^2(\er^d))^m$. The contradiction now follows from \eqref{contra}.
%	$$\int f(-s_n,-y_n)\mathbf{Q}\mathbf{Q}=e^{-is_n\omega}\int \mathbf{Q}(x-y_n)\mathbf{Q}(x)dx=e^{-is_n\omega}\mathbf{Q}*{\widecheck{\mathbf{Q}}}(y_n)\to 0$$
%	since $\mathbf{Q},\widecheck{\mathbf{Q}}\in L^2$. Hence \eqref{contra} implies a contradiction.
\end{proof}

\noindent
We now consider a smooth path $\Gamma\,:t\in[0,\epsilon[\to \Gamma(t)\in \mathcal{M}$ with $\Gamma(0)=\mathbf{Q}$ and $\Gamma'(0)=\mathbf{\Psi}$. We define the projection of the orbital neighborhood onto the orthogonal neighborhood $G: \mathcal{O}_{\mathbf{Q},\delta} \mapsto L_{\delta} $ as
$$ G(\mathbf{w})=\tilde{\mathbf{w}}\in L_{\delta}, \mbox{ where }\mathbf{w}=f(\theta,y)\tilde{\mathbf{w}}, \mbox{ for some }(\theta,y)\in [0,2\pi/\tilde{\omega})\times\er^d.$$
%where $W=f(s,y)\tilde{W}$, $\tilde{W}\in L_{\delta}$.
%$$G\,:\,W\in \mathcal{O}_{\mathbf{Q},\delta}\to G(W)=f(-s,-y)W\in L_{\delta},$$
%where $W=f(s,y)\tilde{W}$, $\tilde{W}\in L_{\delta}$.
Notice that $G$ is well-defined: if $\mathbf{w}=f(\theta_1,y_1)\mathbf{w}_1=f(\theta_2,y_2)\mathbf{w}_2$ with $\mathbf{w}_1,\mathbf{w}_2\in L_\delta$, then $\mathbf{w}_1=f(\theta_2-\theta_1,y_2-y_1)\mathbf{w}_2$ and $(\theta_1,y_1)=(\theta_2,y_2)$ by Lemma \ref{Lema42}.
\vskip10pt
\noindent
We set $$ A(\mathbf{w})=\langle J^{-1}\mathbf{\Psi},G(\mathbf{w})\rangle_{L^2},\quad \mbox{ for }\mathbf{w}\in \mathcal{O}_{\mathbf{Q},\delta}.$$
Notice that, by definition of $L$, 
$$(L^2(\er^d))^m=L\oplus\langle J^{-1}\omega\mathbf{Q}\rangle\oplus_j\langle\partial_{x_j}\mathbf{Q}\rangle.$$
Since $A$ is invariant by the action of $f(s,y)$, it is constant along  $\mathbf{Q}+\langle J^{-1}\omega\mathbf{Q}\rangle=\mathbf{Q}+\langle\partial_\theta f(\theta,y)\mathbf{Q}|_{(0,0)}\rangle$ and $\mathbf{Q}+\langle\partial_{x_j}\mathbf{Q}\rangle=\mathbf{Q}+\langle\partial_{y_j} f(\theta,y)\mathbf{Q}|_{(0,0)}\rangle$. Hence  $A'(\mathbf{Q})=A|_{L}'(\mathbf{Q})$, and, for all $\gamma\in L$, 
$$\langle A|_L'(\mathbf{Q}),\gamma\rangle_{L^2}=\frac d{dt}A(\mathbf{Q}+t\gamma)|_{t=0}=\frac d{dt}\langle J^{-1}\mathbf{\Psi},\mathbf{Q}+t\gamma\rangle_{L^2}|_{t=0}=\langle J^{-1}\mathbf{\Psi},\gamma\rangle_{L^2}.$$
Recalling that $J^{-1}\Psi \in L$, we conclude that $A'(\mathbf{Q})=J^{-1}\mathbf{\Psi}$. Moreover, since, by Assumption 3, $\mathbf{\Psi}\in (H^1(\er^d))^m$, one may easily check that $A\in C^1((H^1(\er^d))^m)$.

\bigskip

\noindent
Consider, for $\mathbf{v}\in \mathcal{O}_{\mathbf{Q},\delta}$, the flow $\mathbf{z}=\Lambda(t)\mathbf{v}$ generated by the pseudo-Hamiltonian system 
\begin{equation}
\label{dual}
\left\{
\begin{array}{llll}
\mathbf{z}_t=JA'(\mathbf{z})\\
\mathbf{z}(0)=\mathbf{v}
\end{array}
\right..
\end{equation}

\begin{nota}
	The mass $M$ is conserved by the flow $\Lambda$. Indeed, notice that for all $\mathbf{u}\in \mathcal{O}_{\mathbf{Q},\delta}$,
	$$0=\frac d{ds} A(e^{i\omega s}\mathbf{u})|_{s=0}=\langle A'(\mathbf{u}),i\omega \mathbf{u}\rangle_{L^2}=\langle A'(\mathbf{u}),-JM'(\mathbf{u})\rangle_{L^2}$$
	and so, taking $\mathbf{u}=\Lambda(t)\mathbf{v}$, 
	$$\frac d{dt} M(\Lambda(t)\mathbf{v})=\langle M'(\Lambda(t)\mathbf{v}),JA'(\Lambda(t)\mathbf{v})\rangle_{L^2}=0.$$
\end{nota}
\vskip10pt
\noindent Set
 $$P\,:\,\mathbf{w}\in \mathcal{O}_{\mathbf{Q},\delta}\to P(\mathbf{w})=\langle H'(\mathbf{w}),JA'(\mathbf{w})\rangle_{H^{-1}\times H^1}.$$
Due to Assumption 1, $P$ is well-defined and is continuous. We have the following result:
\begin{Lemma}\label{Lema43}
For $\mathbf{w}\in \mathcal{O}_{\mathbf{Q},\delta}$, $\delta$ small enough, and for small $t\neq0$, 
\begin{equation}
\label{ineqexpo}
H(\Lambda(t)\mathbf{v})-H(\mathbf{v})<tP(\mathbf{v}).
\end{equation}
\end{Lemma}

\begin{proof}
\noindent
%	By continuity, we only need to show that \eqref{ineqexpo} holds for $\mathbf{v}=\mathbf{Q}$, namely, since $P(\mathbf{Q})=0$, that
%$$H(\Lambda(t)\mathbf{Q})<H(\mathbf{Q}).$$
The Taylor expansion of $H(\Lambda(t)\mathbf{Q})$ at $t=0$ reads
$$H(\Lambda(t)\mathbf{Q})=H(\mathbf{Q})+t\langle H'(\mathbf{Q}),JA'(\mathbf{Q})\rangle_{L^2}+\frac 12t^2(\langle H''(\mathbf{Q})A'(\mathbf{Q}),A'(\mathbf{Q})\rangle_{L^2}+\langle H'(\mathbf{Q}),\Lambda''(0)\rangle_{L^2})+o(t^2)$$
\begin{equation}
\label{iii}
=H(\mathbf{Q})+t\langle H'(\mathbf{Q}),\Psi\rangle_{L^2}+\frac 12t^2(\langle H''(\mathbf{Q})\Psi,\Psi\rangle_{L^2}+\langle H'(\mathbf{Q}),\Lambda''(0)\rangle_{L^2})+o(t^2).
\end{equation}

\medskip
\noindent Since $M$ is conserved by $\Lambda$,
$$\langle M'(\mathbf{Q}), \Psi\rangle_{L^2}=\langle M''(\mathbf{Q})\Psi,\Psi\rangle_{L^2}+\langle M'(\mathbf{Q}),\Lambda''(0)\rangle_{L^2}=0.$$
Adding these terms to \eqref{iii} and recalling that $S'(\mathbf{Q})=0$,
$$H(\Lambda(t)\mathbf{Q})=H(\mathbf{Q})+tP(\mathbf{Q}) +  \frac 12t^2\langle S''(\mathbf{Q})\Psi,\Psi\rangle_{L^2}+o(t^2).$$
Therefore, for some $C>0$,
$$
H(\Lambda(t)\mathbf{Q})-H(\mathbf{Q})-tP(\mathbf{Q})<-Ct^2, \quad t\mbox{ small.}
$$
By continuity,
$$
H(\Lambda(t)\mathbf{u})-H(\mathbf{u})-tP(\mathbf{u})\le-Ct^2, \quad \mathbf{u}\in\mathcal{O}_{\mathbf{Q},\delta},\  t\mbox{ small}
$$
and the Lemma is proved. 
\end{proof}

\noindent
Let $\mathcal{V}$ be the codimension one manifold on which $\mathbf{Q}$ minimizes the action $S$. Observe that $J^{-1}\Psi$ is transverse to the manifold $\mathcal{V}$: on one hand, $\langle S''(\mathbf{Q})J^{-1}\Psi,J^{-1}\Psi\rangle_{L^2}=\langle S''(\mathbf{Q})\Psi,\Psi\rangle_{L^2}<0$;
on the other hand, $\mathbf{Q}$ is a minimum of the action over $\mathcal{V}$, and therefore, considering the projection $J^{-1}\tilde{\Psi}$ of $J^{-1}\Psi$ on the tangent space $T_{\mathbf{Q}}\mathcal{V}$, $\langle S''(\mathbf{Q})J^{-1}\tilde{\Psi},J^{-1}\tilde{\Psi}\rangle_{L^2}\geq 0$.
\vskip10pt\noindent
Set $\mathcal{K}=\mathcal{O}_{\mathbf{Q},\delta}\setminus\mathcal{V}$ and fix $\mathbf{u}\in \mathcal{K}\cap \mathcal{M}$.
We claim that there exists $\epsilon>0$ and $t_{\mathbf{u}}\in ]-\epsilon,\epsilon[$ with 
\begin{equation}
\label{it}
H(\mathbf{Q})<H(\mathbf{u})+t_\mathbf{u}P(\mathbf{u}).
\end{equation}
Indeed, as a consequence of the transversality of $A'(\mathbf{Q})=J^{-1}\Psi$ with respect to $\mathcal{V}$, the Implicit Function Theorem implies that, for $\epsilon>0$ small and $\mathbf{v}\in \mathcal{O}_{\mathbf{Q},\delta}$, there exists $t_{\mathbf{v}}\in ]-\epsilon,\epsilon[$ such that $\Lambda(t_\mathbf{v})\mathbf{v}\in \mathcal{V}$. Since $\mathbf{u}\in \mathcal{K}$, then necessarily $t_{\mathbf{u}}\neq 0$. It follows from Lemma \ref{Lema43} that
\begin{equation}
 H(\Lambda(t_\mathbf{u})\mathbf{u})<H(\mathbf{u})+t_\mathbf{u}P(\mathbf{u}).
\end{equation}
The minimality of $\mathbf{Q}$, together with $M(\Lambda(t_\mathbf{u})\mathbf{u})=M(\mathbf{u})=M(\mathbf{Q})$, implies that
$$
H(\mathbf{Q})+M(\mathbf{Q})=S(\mathbf{Q})\le S(\Lambda(t_\mathbf{u})\mathbf{u}) = H(\Lambda(t_\mathbf{u})\mathbf{u}) = H(\Lambda(t_\mathbf{u})\mathbf{u})+M(\mathbf{Q}).
$$
The claim now follows from the two previous inequalities.
\noindent
%In view of \eqref{ineqexpo} and by the Implicit Function Theorem, we garantee the existence of $\epsilon>0$ and $t_u\in ]-\epsilon;\epsilon[$ such that \eqref{it} holds. In particular, since $\Lambda(t_v)v\in \mathcal{V}$,  $S(\mathbf{Q})\leq S(\Lambda(t_v)v)$ and by conservation of mass by the flow of \eqref{dual},
%$$H(\mathbf{Q})\leq H(\Lambda(t_v)v).$$
%
%\bigskip
%
%\bigskip

\vskip10pt

\noindent 
\begin{proof}[Proof of Theorem \ref{gon}]

Consider  $S_+:=\{\mathbf{u}\in \mathcal{O}_{\mathbf{Q},\delta}\,:\,H(\mathbf{u})<H(\mathbf{Q}), M(u)=M(\mathbf{Q}),P(u)>0\}.$ Since $\mathbf{Q}$ is a local minimum of $S$ over $\mathcal{V}$, $\mathcal{V}\cap S_+=\emptyset$, meaning that $S_+\subset \mathcal{K}\cap \mathcal{M}$. If $u_0\in S_+$, by the conservation of $M$ and $H$ and in view of \eqref{it},
$$0< H(\mathbf{Q})-H(\mathbf{u}(t))<t_{\mathbf{u}(t)}P(\mathbf{u}(t)),\quad t_{\mathbf{u}}\neq 0.$$
Hence $P(\mathbf{u}(t))\neq 0$ and, by continuity, $P(\mathbf{u}(t))>0$, that is, $\mathbf{u}\in S_+$. Hence $S_+$ is conserved by the flow generated by \eqref{sistema}.

We may now conclude has in \cite{SS}: for $\mathbf{u}_0\in S_+$ with $\|\mathbf{u}_0-\mathbf{Q}\|_{(H^1(\er^d))^m}\ll \delta$,
$$\frac{d}{dt}A(\mathbf{u}(t))=\langle \mathbf{u}_t,A'(\mathbf{u}(t))\rangle_{H^{-1}\times H^1}=\langle JH'(\mathbf{u}(t)),J^{-1}\mathbf{u}_t\rangle_{H^{-1}\times H^1}=\langle JH'(\mathbf{u}(t)),A'(\mathbf{u}(t))\rangle_{H^{-1}\times H^1}=P(\mathbf{u}(t)).$$
Since 
$$P(\mathbf{u}(t))>\frac{H(\mathbf{Q})-H(\mathbf{u}_0)}{t_{\mathbf{u}}}>\frac{H(\mathbf{Q})-H(\mathbf{u}_0)}{\epsilon}>0,$$ one has $$\lim_{t\to +\infty}|A(\mathbf{u}(t))|=+\infty.$$
The contradiction follows from
$|A(\mathbf{u}(t))|\leq \|\Psi\|_{L^2}\|G(\mathbf{u}(t))\|_{L^2}\leq \|\Psi\|_{L^2}(\|\mathbf{Q}\|_{H^1}+\delta)$.
\end{proof}

\subsection{Construction of an unstable direction - Proof of Theorem \ref{teo:principal}}
\noindent
In what follows, let us consider a real bound-state $\mathbf{Q}\in \mathcal{B}_0$ and
a smooth path $\Gamma\,:\,[0,\epsilon[\to \mathcal{M}$ given by
$$\Gamma(t)=\Big(\gamma_1(t)\lambda^{\frac d2}(t)Q_1(\lambda(t) x),\dots,\gamma_m(t)\lambda^{\frac d2}(t)Q_m(\lambda(t)x)\Big),$$
with $\Gamma(0)=\mathbf{Q}$ (i.e. $\lambda(0)=\gamma_j(0)=1)$. Our goal is to suitably choose $\lambda, \gamma_j$ so that $\Psi=\Gamma'(0)$ satisfies the conditions of Theorem \ref{gon}. Now:
\begin{itemize}
\item By Assumption 3, $\Psi\in (L^2(\er^d))^m$ and $\langle S''(\mathbf{Q})\Psi,\Psi\rangle_{H^{-1}\times H^1}$ is well-defined.
\item The condition $\Gamma(t)\in\mathcal{M}$, which implies that $\Psi$ is tangent to $\mathcal{M}$ at $\mathbf{Q}$, is equivalent to $\displaystyle \frac{d}{dt}M(\Gamma(t))=0$,
that is,
\begin{equation}
\label{ka1}
\sum_j \gamma_j(t)\gamma_j'(t)\lambda_j\omega_j\int Q_j^2=0.
\end{equation}
%Taking $t=0$,
%\begin{equation}
%\sum_j \gamma_j'(0)n_j^2\int Q_j^2=0.
%\end{equation}
\item The fact that $J^{-1}\Psi$ has complex components immediately implies that $$\langle J^{-1}\Psi,\partial_{x_j}\mathbf{Q}\rangle_{L^2}=0.$$
Furthermore, $\langle J^{-1}\Psi,i\omega\mathbf{Q}\rangle_{L^2}=\langle\Psi,\nabla M(\mathbf{Q})\rangle_{L^2}=0.$
\item Again, since $\mathbf{Q}$ is a real bound-state, $J^{-1}\omega \mathbf{Q}$ is orthogonal to $\partial_{x_j}\mathbf{Q}$.
\end{itemize}
Hence, we only need to see that $\langle S''(\mathbf{Q})\Psi,\Psi\rangle_{H^{-1}\times H^1}<0$. We have
$$\frac{d}{dt}S(\Gamma(t))=\langle S'(\Gamma(t)), \Gamma'(t)\rangle$$ and
$$\frac{d^2}{dt^2}S(\Gamma(t))=\langle S'(\Gamma(t)), \Gamma''(t)\rangle+\langle S''(\Gamma'(t)), \Gamma'(t)\rangle.$$
Evaluating the above equality at $t=0$, since $\mathbf{Q}=\Gamma(0)$ is a bound-state, we get
$$\langle S''(\Psi),\Psi\rangle<0\Leftrightarrow \frac{d^2}{dt^2}S(\Gamma(t))|_{t=0}<0 \Leftrightarrow \frac{d^2}{dt^2}H(\Gamma(t))|_{t=0}<0$$
(recall that $\Gamma\subset \mathcal{M}$).
We now compute the Hamiltonian $H$ along the path $\Gamma$:
\begin{align*}
 H(\Gamma(t))&=\frac 12\sum_j \gamma_j^2(t)\lambda^2(t)\int |\nabla Q_j|^2 +\frac 12\sum_k \lambda^{\frac {d\alpha_k}2-d}(t)\int n_k(\gamma_1(t)Q_1,\dots,\gamma_m(t)Q_m)\\
 &=  \frac 12 \lambda^2(t)\sum_j \gamma_j^2(t)\int |\nabla Q_j|^2 +\frac 12\sum_k \Big(\prod_j \gamma_j^{\beta_{j,k}}(t)\Big)\lambda^{\frac {d\alpha_k}2-d}(t)\int n_k(\mathbf{Q}).
\end{align*}
Differentiating with respect to $t$, we obtain
$$\dfrac{d}{dt}H(\Gamma(t))=\lambda'(t)A(t)+\sum_{j} \gamma_{j}'(t)B_{j}(t)$$
with
\begin{equation}
A(t)=\lambda(t)\sum_j\gamma_j^2(t)\int |\nabla Q_j|^2+ \frac d2\sum_k\Big(\frac{\alpha_k}2-1\Big)\lambda^{\frac{d\alpha_k}2-d-1}(t)\Big(\prod_j\gamma_j^{\beta_{j,k}}(t)\Big)\int n_k(\mathbf{Q})
\end{equation}
and 
\begin{equation}
B_{j_0}(t)=\lambda^2(t)\gamma_{j_0}(t)\int |\nabla Q_{j_0}|^2
+\frac 12\sum_{k}\lambda^{\frac{d\alpha_k}{2}-d}(t)\beta_{j_0,k}\gamma_{j_0}^{\beta_{j_0,k}-1}(t)\Big(\prod_{j\neq j_0}\gamma_j^{\beta_{j,k}}(t)\Big)\int n_k(\mathbf{Q}).
\end{equation}
Since $\mathbf{Q}\neq 0$, we may assume, without loss of generality, that $Q_m\neq 0$. Condition \eqref{ka1} then yields
\begin{equation}
\label{ka2} 
\gamma'_m(t)=-\dfrac{1}{\gamma_m(t)}\sum_{j\neq m}k_j\gamma_j(t)\gamma_j'(t).
\end{equation}
%where $\displaystyle k_j=\dfrac{\lambda_j\omega_j\int Q_j^2}{\lambda_M\omega_M\int Q_M^2}.$\\
Thus,
\begin{equation}
\label{ind}
\dfrac{d}{dt}H(\Gamma(t))=\lambda'(t)A(t)+\sum_{j< m} \gamma_{j}'(t)\Big(B_{j}(t)-\dfrac{1}{\gamma_m(t)}k_j\gamma_j(t)B_m(t)\Big).
\end{equation}
Now, observe that since $\mathbf{Q}$ is a bound-state, $ \frac{d}{dt}H(\Gamma(t))\big|_{t=0}=0$ independently of the choices of $\lambda,\gamma_1,\dots,\gamma_{m-1}$. Hence, 
\begin{equation}\label{eq:primeiraderivadazero}
A(0)=0\quad\textrm{ and }\quad B_j(0)-k_jB_m(0)=0,
\end{equation}
and we easily deduce the following $m$ independent equalities regarding the bound-state $\mathbf{Q}$:
\begin{Prop}[First integrals]
	\label{propgradientes} Let $\mathbf{Q}=(Q_1,\dots Q_m)$ a real bound-state of system \eqref{sistema}, with $Q_m\neq 0$. Define
	$$
	K=\sum_j k_j=\dfrac{1}{\lambda_m\omega_m\int Q_m^2}M(\mathbf{Q}).
	$$ Then, for any $1\le j\le m$,
	$$
	\int |\nabla Q_j|^2 = \frac{1}{2}\sum_k \left[ -\frac{k_j}{K}\left( d\left(\frac{\alpha_k}{2}-1\right) \sum_{l} (k_l\beta_{m,k}-\beta_{l,k}) \right) + k_j\beta_{m,k} - \beta_{j,k}\right]\int n_k(\mathbf{Q})=: \sum_k \sigma_{j,k}\int n_k(\mathbf{Q})
	$$
	In particular,
\begin{equation}
\label{gradientes1}
\sum_j \int |\nabla Q_j|^2=\sum_{j,k} \sigma_{j,k}\int n_k(\mathbf{Q})=-\dfrac{d}{2}\sum_k\Big(\frac{\alpha_k}{2}-1\Big)\int n_k(\mathbf{Q}).
\end{equation}
\end{Prop}

\begin{Rem}
	It is important to observe that
\begin{equation}\label{eq:simplificasigmas}
	\sigma_{j,k}-k_j\sigma_{m,k}=-\frac{1}{2}(\beta_{j,k}-k_j\beta_{m,k}).
\end{equation}
\end{Rem}
%\begin{nota} Notice that Proposition \ref{propgradientes} allows to compute the $L^2$-norms of the gradients $\nabla Q_j$ in terms of the nonlinearities of the potential energy $N(\mathbf{Q})$. Indeed, it is easy to derive that, for all $j=1, \dots, M$,
%	$$\int |\nabla Q_j|^2=\dfrac 1{2\mathcal{K}}\Big(k_jd(1-\frac{\alpha_k}2)+(\mathcal{K}-k_j)(k_j\beta_{M,k}-\beta_{j,k})\Big)\int n_k(\mathbf{Q}),$$
%	where $\displaystyle \mathcal{K}=\sum_j k_j=\dfrac{1}{n_M\int Q_M^2}M(\mathbf{Q}).$
%\end{nota}
We now compute the second derivative at $t=0$:
$$\frac{d^2}{dt^2}H(\Gamma(t))\Big|_{t=0} =\left[ \lambda''A + \lambda'A' + \sum_{j<m}\gamma_j''\left(B_j - \frac{1}{\gamma_m}k_j\gamma_jB_m\right) + \sum_{j<m}\gamma_j'\left(B_j - \frac{1}{\gamma_m}k_j\gamma_jB_m\right)'\right]\Bigg|_{t=0}.
$$
From \eqref{eq:primeiraderivadazero}, the first and third terms are zero, and so
\begin{align*}
\frac{d^2}{dt^2}H(\Gamma(t))\Big|_{t=0} &=\left[ \lambda'A' + \sum_{j<m}\gamma_j'\left(B_j' - k_j\gamma_j'B_m - k_jB_m' + k_j\gamma_m'B_m\right)  \right]\Bigg|_{t=0} \\&= \left[ \lambda'A' + \sum_{j<m}\gamma_j'\left(B_j' - k_j\gamma_j'B_m - k_jB_m' - k_j\sum_{l< m} k_l\gamma_l' B_m\right)  \right]\Bigg|_{t=0}.
\end{align*}
Thus the second derivative is a quadratic form applied to $(\lambda'(0),\gamma_1'(0),\dots, \gamma_{m-1}'(0))$. To simplify the following exposition, we write $\lambda'$ (resp. $\gamma_{j}'$) instead of $\lambda'(0)$ (resp. $\gamma_{j}'(0)$).
\begin{align*}
A'(0)&=\sum_k \left[ \sum_{j}(\lambda' + 2\gamma_{j}') \sigma_{j,k}  +\frac{d}{2} \left(\frac{\alpha}{2}-1\right)\left(\frac{d\alpha_k}{2}-d-1\right)\lambda' +\frac{d}{2}\sum_{j}\left(\frac{\alpha_k}{2}-1\right)\beta_{j,k}\gamma_j'  \right]\int n_k(\mathbf{Q}).
\end{align*}
For a fixed $j_0<m$,
\begin{align*}
B_{j_0}'(0)&=\sum_k \left[(2\lambda'+\gamma_{j_0}')\sigma_{j_0,k} + \frac{d}{2} \left(\frac{\alpha_k}{2}-1\right)\beta_{j_0,k}\lambda' +\frac{1}{2} \gamma_{j_0}'\beta_{j_0,k}(\beta_{j_0,k}-1) + \frac{1}{2}\beta_{j_0,k}\left(\sum_{j\neq j_0} \beta_{j,k}\gamma_j'\right)\right] \int n_k(\mathbf{Q})\\
&=\sum_k \Bigg[(2\lambda'+\gamma_{j_0}')\sigma_{j_0,k} + \frac{d}{2} \left(\frac{\alpha_k}{2}-1\right)\beta_{j_0,k}\lambda' +\frac{1}{2} \gamma_{j_0}'\beta_{j_0,k}(\beta_{j_0,k}-1) \\&\quad + \frac{1}{2}\beta_{j_0,k}\left(\sum_{j\neq j_0,m} \beta_{j,k}\gamma_j'\right) -\sum_{j<m}\frac{1}{2} \beta_{j_0,k}\beta_{m,k}k_j\gamma_j' \Bigg] \int n_k(\mathbf{Q})
\end{align*}
and
\begin{align*}
B_m'(0)=\left[ \left(2\lambda' - \sum_{j< m}k_j\gamma_j'\right)\sigma_{m,k} + \frac{d}{2}\left(\frac{\alpha_k}{2}-1\right)\beta_{m,k}\lambda'  + \frac{1}{2}\beta_{m,k}\sum_{j<m}\left(\beta_{j,k} -  k_j (\beta_{m,k}-1) \right)\gamma_j'   \right]\int n_k(\mathbf{Q}).
\end{align*}
Writing the quadratic form in terms of a symmetric matrix $A$, we now collect the various entries:
\begin{itemize}
	\item $(\lambda')^2$:
	$$
	a_{0,0}= \sum_k\left[ \sum_{j}\sigma_{j,k} +\frac{d}{2} \left(\frac{\alpha}{2}-1\right)\left(\frac{d\alpha_k}{2}-d-1\right) \right]\int n_k(\mathbf{Q})= \sum_k \frac{d}{2} \left(\frac{\alpha}{2}-1\right)\left(\frac{d\alpha_k}{2}-d-2\right) \int n_k(\mathbf{Q})
	$$
	\item $\lambda'\gamma_{j_0}$: using \eqref{eq:simplificasigmas},
\begin{align*}
	a_{0,j_0}&=\sum_k\left[  2(\sigma_{j,k}-k_j \sigma_{m,k}) + \frac{d}{2}\left(\frac{\alpha_k}{2}-1\right)(\beta_{j,k}-k_j\beta_{m,k}) \right] \int n_k(\mathbf{Q}) \\&= \sum_k\left[  \frac{d}{2}\left(\left(\frac{\alpha_k}{2}-1\right)-1\right)(\beta_{j,k}-k_j\beta_{m,k}) \right] \int n_k(\mathbf{Q})
\end{align*}
	\item $(\gamma_{j_0}')^2$:
	\begin{align*}
	a_{j_0,j_0}& = \sum_k \Bigg[ \sigma_{j_0,k} + \frac{1}{2}(\beta_{j_0,k}-1)\beta_{j_0,k} - \frac{1}{2}k_{j_0}\beta_{m,k}\beta_{j_0,k} - k_{j_0}\left(\sigma_{m,k} + \frac{1}{2}\beta_{m,k}\right)\\& + k_{j_0}^2\sigma_{m,k} + \frac{1}{2}k_{j_0}^2\beta_{m,k}(\beta_{m,k}-1) - \frac{1}{2}k_{j_0}\beta_{m,k}\beta_{j_0,k} - \frac{1}{2}\beta_{m,k}k_{j_0}^2 \Bigg]\int n_k(\mathbf{Q})\\&=\sum_k \left[ \frac{k_{j_0}^2}{2}\beta_{m,k}(\beta_{m,k}-2) - k_j\beta_{m,k}\beta_{j_0,k} + \frac{1}{2}\beta_{j_0,k}(\beta_{j_0,k}-2)\right]\int n_k(\mathbf{Q})
	\end{align*}
	\item $\gamma_{j_0}\gamma_{j_1}$, $j_0\neq j_1$:
		\begin{align*}
	a_{j_0,j_1}& = \sum_k \Bigg[ \frac{1}{2}\beta_{j_0,k}\beta_{j_1,k} - \frac{1}{2}k_{j_1}\beta_{m,k}\beta_{j_0,k} + k_{j_0}k_{j_1}\sigma_{m,k} + \frac{1}{2}k_{j_0}k_{j_1} \\&- \frac{1}{2}k_{j_0}\beta_{m,k}\beta_{j_1,k} - k_{j_0}k_{j_1}\sigma_{m,k} - \frac{1}{2}k_{j_0}k_{j_1}\beta_{m,k} \Bigg]\int n_k(\mathbf{Q})
	\end{align*}
\end{itemize}
Hence the symmetric form is represented by the matrix $A$ given in the statement of Theorem \ref{teo:principal}. Since, by assumption, $A$ has a negative eigenvalue, there exists a nontrivial choice of $(\lambda'(0), \gamma_1'(0), \dots, \gamma_{m-1}'(0))$ such that
$$
\langle S''(\mathbf{Q})\Gamma'(0), \Gamma'(0)\rangle= \frac{d^2}{dt^2}H(\Gamma(t))\Big|_{t=0}<0.
$$
This concludes the proof of Theorem \ref{teo:principal}.\hfill\qed

\begin{proof}[Proof of Proposition \ref{ucoro}]
	Just notice that
	\begin{align*}
a_{0,0}&=\dfrac d2\sum_k\Big(\dfrac{\alpha_k}2-1\Big)\Big(\dfrac{d\alpha_k}2-d-2\Big)\int n_k(\mathbf{Q})\\&\le \dfrac d2\Big(\dfrac{dp}2-d-2\Big)\sum_k\left(\frac{\alpha_k}{2}-1\right)\int n_k(\mathbf{Q})=-\Big(\dfrac{dp}2-d-2\Big)\sum_j\|\nabla Q_j\|_2^2<0
	\end{align*}
	from Proposition \ref{propgradientes}.
\end{proof}

%NOTA: Este resultado é verdadeiro se existirem mais termos: homogeneidades acima de p e com sinal negativo ou abaixo de p, com sinal positivo.
%
%\begin{Prop}[$L^2$-critical instability I]\label{L2critical}
%Suppose that
%$$N(\mathbf{u})=N_2(\mathbf{u})+N_p(\mathbf{u}),$$
%where 
%$$N_2(\mathbf{u})=\sum_{j=1}^M c_j \int |u_j|^2 $$
%and $N_p$ is homogeneous of degree $p=2+4/d$. If $c_jm_M\neq c_Mm_j$, for some $j<M$, then all real bound-states $\mathbf{Q}\in \mathcal{B}_0$ with $Q_j, Q_M\neq 0$ are orbitally unstable.
%\end{Prop}
\begin{proof}[Proof of Proposition \ref{L2critical}]
	The specific homogeneities of $N$ imply that $a_{0,0}=0$. Furthermore,
	$$
	a_{0,1}=-2(c_1\int Q_1^2 - k_1c_m\int Q_m^2) = -2\left(c_1-\frac{c_m\lambda_1\omega_1}{\lambda_m\omega_m} \right)\int Q_1^2 \neq 0.
	$$
	Hence the principal minor
	$$
	A_1=\left[\begin{array}{cc}
	a_{0,0} & a_{0,1}\\ a_{0,1} & a_{1,1}
	\end{array}\right]
	$$
	has negative determinant and $A$ cannot be semi-positive definite.
\end{proof}
%
%\begin{Prop}[$L^2$-critical instability II]\label{L2critical2}
%	Suppose that
%	$$N(\mathbf{u})=N_2(\mathbf{u})+N_p(\mathbf{u}),$$
%	where 
%	$$N_2(\mathbf{u})= \lambda\partere \int u_1\bar{u}_M,\quad \lambda\neq 0. $$
%	and $N_p$ is homogeneous of degree $p=2+4/d$. If $\mathbf{Q}\in \mathcal{B}_0$ satisfies $m_1\int Q_1^2 \neq m_M\int Q_M^2$ and $\int Q_1Q_M\neq 0$, then $\mathbf{Q}$ is orbitally unstable.
%\end{Prop}
\begin{proof}[Proof of Propostion \ref{L2critical2}]
	As in the previous proof, it suffices to check that $a_{0,0}=0$ and $a_{0,1}\neq 0$, which is a direct consequence of the hypothesis.
\end{proof}

%
%
%As a trivial consequence of the previous result, one may obtain an instability result in $L^2$-subcritical cases, which states that even ground-states may be unstable. 
%\begin{Coro}[Instability in almost $L^2$-critical cases]
%	Suppose that 
%	$$N(\mathbf{u})=N_2(\mathbf{u})+N_p(\mathbf{u}),$$
%	where 
%	$$N_2(\mathbf{u})=\sum_{j=1}^M c_j \int |u_j|^2 $$
%	and $N_p$ is homogeneous of degree $p$. Moreover, assume that there exists a continuous curve of real bound-states $p\mapsto \mathbf{Q}(p)\in \mathcal{B}_0$ around $p_0=2+4/d$. If the hypothesis of Proposition \ref{L2critical} are verified for $p_0$, then all  bound-states $\mathbf{Q}(p)$,  $|p-(2+4/d)|\ll 1$,  are orbitally unstable.
%\end{Coro}

\section{Applications}

Before we begin to study the examples stated in the Introduction, it will be useful to recall some facts regarding synchronous systems. Suppose that
$$N(\mathbf{u})=N_2(\mathbf{u})+N_p(\mathbf{u}),$$
where 
$$N_2(\mathbf{u})=\sum_{j=1}^m c_j \int |u_j|^2, \quad N_p(\mathbf{u})=-\int f(\mathbf{u}) $$
with $f:\ce^m\to \er$, homogeneous of degree $p$, such that $f(X)\le f(|X|)$, for all $X\in \ce^m$. By synchronicity, we mean that
$$
c_j+\lambda_j\omega_j\equiv c+\lambda\omega >0 ,\quad j=1,\dots, m.
$$
Define
$$
f_{max}=\sup_{X\in \mathbb{S}^{m-1}} f(X).
$$
$$
\mathcal{X} = \{ X\in \mathbb{S}^{m-1} :  f(X)=f_{max} \}.
$$
and consider the scalar equation
\begin{equation}\label{eq:bs_escalar}
-(c+\lambda\omega)u + \Delta u + a u^{p-1}=0.
\end{equation}
\begin{Prop}\label{prop:caractgs}
	If $X_0\in \er^m$ is a critical point of $f$ on the sphere and $u$ is a solution of \eqref{eq:bs_escalar} with $a=pf(X_0)/2$, then $X_0u$ is a bound-state of \eqref{sistema}. Furthermore, the set of ground-states is given by
	$$
	G=\left\{ Xq: X\in \mathcal{X},\ q \mbox{ ground-state of \eqref{eq:bs_escalar} with }  a=pf_{max}/2 \right\}.
	$$
\end{Prop}
\begin{proof}
Observe that, for some $\gamma\in \er$, $\nabla f(X_0)=\gamma X_0$. Since $f$ is homogeneous of degree $p$,
$$
pf(X_0) = \nabla f(X_0)\cdot X_0 = \gamma |X_0|^2 = \gamma.
$$
One may now check that $X_0u$ satisfies the elliptic system for the bound-states. The characterization of $G$ follows the exact same argument as in \cite{Correia}.
\end{proof}

\subsection{Quadratic Schrödinger system I}
We recall \eqref{quadeq}:
	\begin{displaymath}
	\left\{\begin{array}{llll}
	iu_t+\Delta u+\overline{u}v=0\\
	i\sigma v_t+\Delta v-\beta v+\dfrac 12 u^2=0,\quad (x,t)\in \er^d\times \er.
	\end{array}\right.
	\end{displaymath}	
	Here,
	$$N(u,v)= \beta\int |v|^2-Re\int \overline{u}^2v,$$
	$$n_{1}=\beta |v|^2\quad (\beta_{1,1}=0,\,\beta_{2,1}=2),\quad n_{2}=-Re \beta \overline{u}^2v\quad (\beta_{1,2}=2,\,\beta_ {2,2}=2)$$
	and
	$$M(u,v)=\int |u|^2+2\sigma\int |v|^2,\quad k_1=\frac{\int Q_1^2}{2\sigma\int Q_2^2}.$$ 
	The associated matrix defined in Theorem \ref{teo:principal} is
%	We obtain
%	\begin{displaymath}
%	\begin{array}{llll}
%	H''(0)&=&\displaystyle\lambda'^2\dfrac d8\Big(4-d\Big)\int Q_1^2Q_2\\
%	&&\displaystyle +2\lambda'\gamma_1'\Big(2k_1\beta\int Q_2^2+\frac 14(d-4)(k_1-2)\int Q_1^2Q_2\Big)\\
%	&&\displaystyle +\gamma'^2\frac{k_1}2(k_1+4)\int Q_1^2Q_2,
%	\end{array}
%	\end{displaymath}
%	that is
	\begin{displaymath}
	A=\left[\begin{array}{ccc}
	\displaystyle\dfrac d8\Big(4-d\Big)\int Q_1^2Q_2&\displaystyle 2k_1\beta\int Q_2^2+\frac 14(d-4)(k-2)\int Q_1^2Q_2\\
	&\\
	\displaystyle 2k_1\beta\int Q_2^2+\frac 14(d-4)(k_1-2)\int Q_1^2Q_2&\displaystyle\frac{k_1}2(k_1+4)\int Q_1^2Q_2
	\end{array}\right]
	\end{displaymath}
%	We may recover the instability results that were derived in \cite{CCFD} for the $L^2$-(super)critical cases:
%	for $d\geq 5$, $Q$ is orbitally unstable (Corollary \ref{ucoro});
%	if $d=4$, we get the same conclusion as long as $\beta\neq 0$ (Corollary \ref{L2critical}).
	
%	Now we consider the $L^2$-subcritical cases $d\le 3$ and prove that, under specific conditions on the parameters, one observes the instability of ground-states. 
	\begin{proof}[Proof of Proposition \ref{prop:subL2quad2D}]
	For a given frequency $\omega>0$, one may look for bound-states of the form $(u,v)=(e^{i\omega t}P, e^{2i\sigma\omega t}Q)$. The corresponding stationary system is
	\begin{displaymath}
	\left\{\begin{array}{llll}
	-\omega P+\Delta P+\overline{P}Q=0\\
	-(2\omega\sigma+\beta) Q+\Delta Q+\dfrac 12 P^2=0
	\end{array}\right.
	\end{displaymath}	
	Observe that, if $\beta=\omega(1-2\sigma)$, the system is synchronous. Thus, by Proposition \ref{prop:caractgs}, the ground-state can be computed explicitly:
	%
	%Now we show that, in dimension $d=1$ and for $\sigma = 1-\beta$ large, the corresponding ground-state is orbitally unstable. First, from Proposition \ref{prop:caractgs}, the ground-state is given by
	$$
	(P,Q)= \left(\sqrt{\frac{2}{3}}q,  \frac{1}{\sqrt{3}}q\right),
	$$
	where $q$ is the ground-state of
	$$
	-\omega q+\Delta q + \frac{1}{\sqrt{3}}q^2=0.
	$$
	Consequently,
	$$
	k_1=\frac{1}{\sigma}, \quad k_1\beta=\frac{1-2\sigma}{\sigma}\omega, \quad \int Q_1^2Q_2 = \frac{2}{3\sqrt{3}}\int q^3,\quad \int Q_2^2 = \frac{1}{3} \int q^2.
	$$
	Furthermore, one can see (cf. \cite[Corollary 8.1.3]{cazenave}) that
	$$
	\int q^2 = \frac{6-d}{6\sqrt{3}\omega}\int q^3.
	$$
	Since $a_{0,0}>0$ in all $L^2$-subcritical cases, the condition for instability reduces to $\det(A)<0$:
	$$
	\frac{d(4-d)(1+4\sigma)}{108\sigma^2} - \left( \frac{(1-2\sigma)(6-d)}{9\sqrt{3}\sigma} + \frac{(d-4)(1-2\sigma)}{6\sqrt{3}\sigma}  \right)^2<0,\mbox{ that is, }(2\sigma-1)^2>3(4-d). 
	$$
	Since $\sigma>0$, we observe orbital instability when $\sigma>\sigma_0$.
	\end{proof}
\subsection{Quadratic Schrödinger system II}

	\begin{equation}
	\label{ademir}
	\left\{\begin{array}{llll}
	3iu_t+\Delta u-\beta_1u=-vw\\
	2iv_t+\Delta v-\beta v=-\dfrac 12w^2-u\overline{w}\\
	iw_t+\Delta w-w=-v\overline{w}-u\overline{v},\quad (x,t)\in \er\times \er.
	\end{array}\right.
	\end{equation}

	\noindent
	In this case, $(\omega_1,\omega_2,\omega_3)=(3,2,1)$, $(\lambda_1,\lambda_2,\lambda_3)=(3,2,1)$ and
%	The mass reads
%	$$M(u,v,w)=9\int |u|^2+4\int|v|^2+\int|w|^2.$$
%	Also,
	$$N(u,v,w)=\beta_1\int |u|^2+\beta\int |v|^2+ \int |w|^2-Re\int \overline{w}^2v-2\int \overline{u}vw.$$
	
	\noindent	
	The synchronicity is obtained for $\beta_1=-7$ and $\beta=-2$, and, in this situation,
	$$N_2(u,v,w)=-7\int |u|^2-2\int |v|^2+\int |w|^2\textrm{ and }N_3(u,v,w)=-Re\int \overline{w}^2v-2\int \overline{u}vw.$$
	\begin{proof}[Proof of Proposition \ref{prop:subL2quad3D}]
We apply Proposition \ref{prop:caractgs}: setting $f(x,y,z)=  yz^2+2xyz,$ elementary computations show that the maximum is $f_{max}=(\sqrt{3}+\sqrt{5})/9$ and that it is achieved at $$(a,b,c)=\left(\sqrt{\frac 1{15}(5-\sqrt{5})},\dfrac 1{\sqrt{3}},\sqrt{\frac 1{15}(5+\sqrt{5})}\right)\quad \mbox{and }(-a,b,-c).$$
	Hence the set of ground-states is given by
	$$
	G=\{(aQ,bQ,cQ), (-aQ,bQ,-cQ)\} \mbox{ modulo translations and rotations,}
	$$
	where $Q$ satisfies
%	Hence, $(aQ,bQ,cQ)$ is a ground-state of \eqref{ademir}, with
	\begin{equation}
	\label{q1}
	-2Q+\Delta Q+3f_{max}Q^2=0.
	\end{equation}
%	Furthermore, 
%	$n_{\alpha_1}=-7|u|^2\,(\beta_{1,1}=2,\beta_{2,1}=\beta_{3,1}=0)$, $n_{\alpha_2}=-2|v|^2\,(\beta_{2,2}=2,\beta_{1,2}=\beta_{3,2}=0)$,\\
%	$n_{\alpha_3}=|w|^2\,(\beta_{3,3}=2,\beta_{1,3}=\beta_{2,3}=0)$,
%	$n_{\alpha_4}=-v\overline{w}^2\,(\beta_{1,4}=0,\beta_{2,4}=1, \beta_{3,4}=2)$ and\\ $n_{\alpha_5}=-2\overline{u}vw\,(\beta_{1,5}=\beta_{2,5}=\beta_{3,5}=1)$.\\
For either ground-state, the entries of matrix $A$ read
	$$a_{0,0}=\dfrac {3}4f_{max}\int Q^3;\quad a_{1,1}=abc(k_1+1)^2\int Q^3;\quad a_{2,2}=bc\Big(\Big(2k_2+\dfrac12\Big)c+(k_2+1)^2a\Big)\int Q^3,$$
	\begin{displaymath}
	\begin{array}{llll}
	a_{0,1}&=&(14a^2+2k_1c^2)\int Q^2-\dfrac34(2k_1bc^2+(k_1-1)abc)\int Q^3\\
	a_{0,2}&=&(4b^2+2k_2c^2)\int Q^2-\dfrac 34\Big((2k_2-1)bc^2+2(k_2-1)abc\Big)\int Q^3\\
	a_{1,2}&=&\Big(bc^2k_1+abc(k_1+k_2+k_1k_2-1)\Big)\int Q^3.
	\end{array}
	\end{displaymath}
%	\medskip
%	
%	\noindent
%	Now, multiplying \eqref{q1} by $Q$ and integrating, we obtain
%	\begin{equation}
%	\label{q2}
%	2\int Q^2+\int |\nabla Q|^2-3f_{max}\int Q^3=0.
%	\end{equation}
%	Also, multiplying \eqref{q1} by $xQ$ and integrating by parts, we get the Pohozaev's identity
%	\begin{equation}
%	\label{q3}
%	2\int Q^2-\int |\nabla Q|^2-2f_{max}\int Q^3=0.
%	\end{equation}
%	From $\eqref{q2}$ and $\eqref{q3}$, $\displaystyle \int Q^2=\dfrac 54f_{max}\int Q^3$.
%	
	Once again, by \cite[Corollary 8.1.3]{cazenave}, 
	$$\displaystyle \int Q^2=\dfrac 54f_{max}\int Q^3.$$
	Finally, since $k_1=\dfrac{9a^2}{c^2}$ and $k_2=4\dfrac{b^2}{c^2}$, putting $a_{i,j}=\tilde{a}_{i,j}\int Q^3$,
	\begin{displaymath}
	\begin{array}{lllll}
	\tilde{a}_{0,0}&=&\dfrac 34f_{max}\\
	\tilde{a}_{1,1}&=&\dfrac{ab(9a^2+c^2)^2}{c^3}\\
	\tilde{a}_{2,2}&=&f_{max}+\dfrac{b(4b^2+c^2)^2}{c^3}\\
	\tilde{a}_{0,1}&=&\dfrac{115}4a^2f_{max}-\dfrac{15}{16}f_{max}\Big(18a^2b+\dfrac{ab(9a^2+c^2)}{c}\Big)\\
	\tilde{a}_{0,2}&=&15b^2f_{max}-\dfrac34\Big(8b^3-bc^2+2\dfrac{ab(4b^2-c^2)}{c}\Big)\\
	\tilde{a}_{1,2}&=&9a^2b+\dfrac{ab}{c}\Big(9a^2+4b^2-c^2+\dfrac{36a^2b^2}{c^2}\Big).
	\end{array}
	\end{displaymath}
	With these values, we get $\det(A)<0$, and therefore both ground-states are orbitally unstable. 
\end{proof}
\appendix

\section{Invariants, Virial identities and blow-up}

In this section, we formally deduce some identities regarding \eqref{sistema} using the Hamiltonian structure. Suppose that one wishes to understand the evolution of a functional $G$ through the trajectories of \eqref{sistema}:
$$
\frac{d}{dt} G(\mathbf{u}(t)) = \langle G'(\mathbf{u}(t)), \mathbf{u}_t\rangle = \langle G'(\mathbf{u}(t)), JH'(\mathbf{u}(t))\rangle =: P(\mathbf{u}(t)).
$$
Consider the Hamiltonian system generated by $G$,
\begin{equation}\label{eq:aux}
\mathbf{v}_t = JG'(\mathbf{v}).
\end{equation}
and prescribe an initial condition $\mathbf{v}_0$. Then
$$
\frac{d}{dt} H(\mathbf{v}(t)) = \langle H'(\mathbf{v}(t)), \mathbf{v}_t\rangle = \langle H'(\mathbf{v}(t)), JG'(\mathbf{v}(t))\rangle = - \langle G'(\mathbf{v}(t)), JH'(\mathbf{v}(t))\rangle = -P(\mathbf{v}(t)).
$$ 
Taking $t=0$, we obtain an alternative definition for $P$:
$$
P(\mathbf{v}_0)=-\frac{d}{dt}H(\mathbf{v}(t))\Big|_{t=0}.
$$
Therefore, the variation of $G$ along the trajectories generated by $H$ is symmetric to the variation of $H$ along the trajectories generated by $G$ at the same state. This duality corresponds to the symmetry of the 
Poisson bracket in Hamiltonian mechanics. The advantage of this formulation is that the dynamical system \eqref{eq:aux} is usually explicitly solvable and the computation becomes trivial.

\begin{Prop}[Conservation of mass and energy]Regarding the flow generated by \eqref{sistema},
	\begin{enumerate}
		\item the Hamiltonian $H$ is conserved; 
		\item  the mass
		$$M(\mathbf{u}(t))=\frac 12 \sum_j\int  \lambda_j\omega_j|u_j(t)|^2$$
		is conserved.
	\end{enumerate}
\end{Prop}
\begin{proof}
	The conservation of $H$ is trivial: one takes $G=H$ and obtains $P=-P$. For the conservation of mass, observe that the dynamical system generated by $M$ is
	$$
	(v_j)_t = -i\omega_j v_j,\quad j=1,\dots,m.
	$$
	The solution of the IVP $\mathbf{v}(0)=\mathbf{v}_0$ is
	$$
	v_j(t) = e^{-i\omega_jt}(v_0)_j,\quad t\in \er, j=1,\dots,m.
	$$
	The invariance \eqref{invmass} implies that $H$ is constant along these trajectories, and thus $P\equiv 0$.
\end{proof}
%
%
%\begin{Prop}
%Under assumption \eqref{invmass}, tha mass
%$$M(\mathbf{u}(t))=\frac 12 \sum_j m_j|u_j(t)|^2, \quad m_j=\lambda_j\omega_j$$
%is conserved by the flow of \eqref{sistema}.
%	\end{Prop}
%\begin{proof}
%Let $\displaystyle f(t):=H(u_1e^{i\omega_1t},\dots,u_Me^{i\omega_Mt})$. It follows from \eqref{invmass} that 
%$$0\equiv f'(t)=\sum_j\langle  i\omega u_je^{i\omega_j t},\frac{\partial H}{\partial u_j}(u_1e^{i\omega_1t},\dots,u_Me^{i\omega_Mt})\rangle,$$
%and, taking $t=0$, $\displaystyle \sum_j \langle i\omega u_j,\frac{\partial H}{\partial u_j}(\mathbf{u}) \rangle=0$. Hence,
%$$\frac d{dt}M(\mathbf{u}(t))=\sum_j \langle m_j u_j,{u_j}_t\rangle=\sum_j \langle m_ju_j, \frac 1{i\lambda_j}\frac{\partial H}{\partial u_j}(\mathbf{u})\rangle=\sum_j \langle i\omega u_j,\frac{\partial H}{\partial u_j}(\mathbf{u}) \rangle=0.$$
%	\end{proof}

\begin{Prop}[Virial identities]
	Define
	$$
	V(\mathbf{u})=\frac{1}{2} \sum_j \int \lambda_j\omega_j|x|^2|u|^2 dx.
	$$
	%Suppose that $m_j\ge 0$, $1\le j\le m$, and 
	Define the space
	$$
	\Sigma=\left\{ \mathbf{u}\in (H^1(\er^d))^m: V(\mathbf{u})<\infty \right\},
	$$
	equipped with the natural norm. Given $\mathbf{u}_0\in \Sigma$, the corresponding solution $\mathbf{u}$ of \eqref{sistema} belongs to $C([0,T(\mathbf{u}_0)), \Sigma)$ and
	\begin{equation}\label{eq:virialprim}
	\frac{d}{dt}V(\mathbf{u}(t))= 2\Im \sum_{j} \int \omega_j \overline{u_j(t)}x\cdot \nabla u_j(t) dx.
	\end{equation}
	Furthermore, if $\omega_j/\lambda_j\equiv \omega_0/\lambda_0$,
	\begin{equation}\label{eq:virialseg}
	\frac{d^2}{dt^2}V(\mathbf{u}(t))= \frac{8\omega_0}{\lambda_0}H(\mathbf{u}_0) - \frac{\omega_0}{\lambda_0}\sum_k (4-N(\alpha_k-2))\int n_k(\mathbf{u}(t)).
	\end{equation}
\end{Prop}
\begin{proof}
	We present only the formal computation.
	The rigorous justification of $\mathbf{u}\in C([0,T(u_0), \Sigma)$ follows from the same arguments as for the single (NLS) equation (see, for example, \cite[Chapter 6]{cazenave}).
	The solution of the initial value problem
	$$
	\mathbf{v}_t=JV'(\mathbf{v}), \quad \mathbf{v}(0)=\mathbf{v}_0
	$$
	is $$
	v_j(t)=e^{-i\omega_j|x|^2t}(v_0)_j, \quad j=1,\dots, m.
	$$
	The invariance \eqref{invmass} implies that
	$$
	\frac{d}{dt}H(\mathbf{v}(t))=\frac{1}{2}\frac{d}{dt}\sum_j \int |\nabla (v_0)_j - 2i\omega_jt x (v_0)_j|^2 dx
	$$
	The identity \eqref{eq:virialprim} follows from the evaluation at $t=0$.
	\vskip10pt
	For the second derivative, one must solve the auxiliary dynamical system associated with the right-hand side of \eqref{eq:virialprim},
	$$
	(v_j)_t = \frac{4\omega_j}{\lambda_j}x\cdot \nabla v_j - \frac{2d\omega_j}{\lambda_j}v_j, \ v_j(0)=(v_0)_j,\quad j=1,\dots, m.
	$$
	Since $\omega_j/\lambda_j=\omega_0/\lambda_0$, the explicit solution is given by
	$$
	v_j(t)=e^{-\frac{2d\omega_0}{\lambda_0}t}(v_0)_j( e^{-\frac{4\omega_0}{\lambda_0}t}x),\quad j=1,\dots, m.
	$$
	Hence, using a change of variables,
	$$
	H(\mathbf{v}(t))= \frac{1}{2}\sum_j\int e^{-\frac{8\omega_0}{\lambda_0}t} |\nabla (v_0)_j|^2 dx +\frac{1}{2} \sum_k \int e^{(4-2\alpha_k)\frac{d\omega_0}{\lambda_0}t}n_k(\mathbf{v}_0)dx
	$$
	Differentiating in time and taking $t=0$, one arrives at the identity \eqref{eq:virialseg}.
\end{proof}

\begin{nota}
	If $\mathbf{u}=e^{i\omega t}\mathbf{Q}$ is a bound-state, then one obtains the Pohozaev identity
	$$
	0 = \frac{\lambda_0}{\omega_0}\frac{d^2V(\mathbf{u}(t))}{dt^2} = H(\mathbf{Q}) - \sum_k (4-d(\alpha_k-2))\int n_k(\mathbf{Q}).
	$$
	If $\alpha_k=2+4/d$ for all $k$, one sees that $H(\mathbf{Q})=0$.
\end{nota}

As a direct consequence from the Virial identities, one has the standard blow-up result using Glassey's argument:
\begin{Prop}[Blow-up]
	Suppose that $\omega_j/\lambda_j=\omega_0/\lambda_0$ and that
	\begin{itemize}
		\item $n_k\le 0$ for $\alpha_k>2+4/d$;
		\item $n_k\ge 0$ for $\alpha_k<2+4/d$.
	\end{itemize}
	Then any initial data $\mathbf{u}_0\in \Sigma$ with negative energy gives rise to a solution which blows-up in finite-time. In particular, if $\alpha_k=2+4/d$ for all $k$, any bound-state is orbitally unstable.
\end{Prop}

\begin{nota}
	Even in the scalar case, the instability of bound-states through a blow-up argument does not work when there are terms in the Hamiltonian which are not $L^2$-critical. In the $L^2$-supercritical case, one may still obtain instability for ground-states, by proving that they are the minimizers of a certain minimization problem. For general nonlinearities and/or bound-states, instability by blow-up remains a challenging open problem.
\end{nota}

\section{Acknowledgements}
S. Correia was partially supported by Funda\c{c}\~ao para a Ci\^encia e Tecnologia, through the grant UID/MAT/04561/2019. J. Drumond Silva was partially supported by Funda\c{c}\~ao para a Ci\^encia e Tecnologia, through the grant UID/MAT/04459/2019. F. Oliveira was partially supported by Funda\c{c}\~ao para a Ci\^encia e Tecnologia, through the grant UID/MULTI/00491/2019.

\bibliography{Biblioteca}
\bibliographystyle{plain}
\begin{center}
	{\scshape Simão Correia}\\
	{\footnotesize
		Centro de Matemática, Aplicações Fundamentais e Investigação Operacional,\\
		Department of Mathematics,\\
		Instituto Superior T\'ecnico, Universidade de Lisboa\\
		Av. Rovisco Pais, 1049-001 Lisboa, Portugal\\
		simao.f.correia@tecnico.ulisboa.pt
	}
\vskip15pt
	{\scshape Filipe Oliveira}\\
{\footnotesize
	Mathematics Department and CEMAPRE\\
	ISEG, Universidade de Lisboa,\\
	Rua do Quelhas 6, 1200-781 Lisboa, Portugal\\
	foliveira@iseg.ulisboa.pt
}
\vskip15pt
	
	{\scshape Jorge D. Silva}\\
	{\footnotesize
		Center for Mathematical Analysis, Geometry and Dynamical Systems,\\
		Department of Mathematics,\\
		Instituto Superior T\'ecnico, Universidade de Lisboa\\
		Av. Rovisco Pais, 1049-001 Lisboa, Portugal\\
		jsilva@math.tecnico.ulisboa.pt
	}

\end{center}

%
% 
%\begin{thebibliography}{x89}
%	\bibitem{q10}{}
%	\bibitem{q11}{}
%	\bibitem{q12}{}
%	\bibitem{Angulo} Angulo, J., Pastor, A., Stability of periodic optical solitons for a nonlinear Schr\"odinger system, Proc. Royal Society of Edinburgh 139A, 927-959, 2009.
%	\bibitem{G} Gon\c{c}alves Ribeiro J.M., Instability of symmetric stationary states for some nonlinear Schr\"odinger equations with an external magnetic field, Annales de l'I.H.P. 54, 403-433, 1991.
%	\bibitem{CCFD} Correia, S., Oliveira, F., Silva, J.D., On a nonlinear Schr\"odinger System arizing in nonlinear optics, to appear in Communications in Mathematical Sciences, 2019.
%	\bibitem{FA} Oliveira, F. and Pastor, A., On a Schr\"odinger System arizing in nonlinear optics, arXiv:1810.08231, 2018.
%	\bibitem{Saut}Colin M., Di Menza L. and Saut, J.C., Solitons in quadratic media, Nonlinearity 29, pp. 1000-1035, 2016.
%	\bibitem{SS}{J. Shatah and W. Strauss, Instability of nonlinear bound-states, Comm. Math. Phys., 100, 173-190 (1985)}
%	\bibitem{Kivshar}
%	\bibitem{Pastor}
%	\bibitem{Sammut}  A. R. Sammut, A. V. Buryak and Y. S. Kivshar. Bright and dark solitary waves in the presence of the third harmonic generation. J. Opt. Soc. Am. B 15, 14881496, 1998. 
%\end{thebibliography}
\end{document}